\documentclass{article}
\usepackage{amssymb}
\usepackage{amsmath}
\usepackage{amsfonts}
\usepackage{bbm}
\usepackage{mathrsfs}
\usepackage{tikz}
\usepackage{pgflibraryarrows}
\usepackage{pgflibraryplotmarks}
\usepackage{hyperref}

\usepackage[letterpaper]{geometry}

\newcommand{\R}{\mathbb{R}}
\newcommand{\E}{\mathbb{E}}
\newcommand{\N}{\mathbb{N}}

\newcommand{\p}{\mathbb{P}}

\newcommand{\M}{\mathbb{M}}
\renewcommand{\P}{\mathbb{P}}
\newcommand{\s}{\sigma}
\renewcommand{\t}{\tau}

\newtheorem{theorem}{Theorem}[section]
\newtheorem{remark}[theorem]{Remark}

\newtheorem{lemma}[theorem]{Lemma}
\newtheorem{definition}[theorem]{Definition}
\newtheorem{corollary}[theorem]{Corollary}

\newenvironment{proof}{{\bf Proof:}}{\hfill$\square$\vskip.5cm}
\newenvironment{proofof}{}{\hfill$\square$\vskip.5cm}

\renewcommand{\M}{\mathbb{M}}
\newcommand{\bc}{\boldsymbol{c}}
\newcommand{\bJ}{\boldsymbol{J}}
\newcommand{\bI}{\boldsymbol{I}}
\newcommand{\bK}{\boldsymbol{K}}
\newcommand{\ba}{\boldsymbol{\alpha}}

\renewcommand{\p}{\mathfrak{p}}
\newcommand{\bxi}{\boldsymbol{\xi}}
\newcommand{\bm}{\boldsymbol{m}}

\def\be{\begin{equation}}
\def\ee{\end{equation}}
\def\bea{\begin{eqnarray}}
\def\eea{\end{eqnarray}}
\def\<{\langle}
\def\>{\rangle}

\title{Antiferromagnetic Potts model \\
on the Erd{\H o}s-R\'enyi random graph\thanks{\href{mailto:pierluigi.contucci@unibo.it}{pierluigi.contucci@unibo.it},
\href{mailto:S.Dommers@tue.nl}{s.dommers@tue.nl},
\href{mailto:cristian.giardina@unimore.it}{cristian.giardina@unimore.it},
\href{mailto:sstarr@math.rochester.edu}{sstarr@math.rochester.edu}}}

\author{Pierluigi Contucci\footnote{Universit\`a di Bologna, Piazza di Porta S.Donato 5, 40127 Bologna, Italy} \and Sander Dommers\footnote{Eindhoven University of Technology, Department of Mathematics and Computer Science, P.O. Box 513, 5600 MB
Eindhoven, The Netherlands} \and Cristian Giardin\`a\footnote{Universit\`a di Modena e Reggio E., viale A. Allegri, 9 - 42121 Reggio Emilia , Italy} \and Shannon Starr\footnote{University of Rochester, Department of Mathematics, Rochester, NY, 14627, USA}}

\date{\today}

\begin{document}

\maketitle

\begin{abstract}
\noindent
We study the antiferromagnetic Potts model on the Poissonian Erd{\H o}s-R\'enyi random graph.
By identifying a suitable interpolation structure and an extended variational
principle, together with a positive temperature second-moment analisys  we
prove the existence of a phase transition at a positive critical temperature. 
Upper and lower bounds on the temperature critical value are
obtained from the stability analysis of the replica symmetric solution
(recovered in the framework
of Derrida-Ruelle probability cascades)
and from an entropy  positivity argument.
%

\vspace{8pt}
\vskip .2 cm
\noindent
{\small \bf MCS numbers:} Primary 60B10, 60G57, 82B20; Secondary 60K35.
\end{abstract}

\maketitle

\section{Introduction and main results}

In this paper we prove some rigorous results on the antiferromagnetic
$q$-Potts model on the Poissonian Erd{\H o}s-R\'enyi random graph of parameter $c$. 
This model is related to diluted spin glasses of disordered statistical mechanics
on one side and to the the graph coloring combinatorial problem on the other.
Since the Erd{\H o}s-R\'enyi random graph has a locally tree-like structure and large loops,
the statistical mechanics model with antiferromagnetic interactions has been reported
to display some spin glass behavior in the physics literature
\cite{van2002random}. In particular it has been argued that the one-step replica symmetry 
breaking solution does not get improved by a higher number of steps \cite{KZ}.
On the other hand it is well known that antiferromagnetic Potts models on graphs are related,
at zero temperature, to the graph coloring problem. This consists in placing
colors on the graph vertices in such a way that two of them connected by
an edge have different color. Some mathematical analysis, from the combinatorial perspective, has been 
obtained in \cite{AchlioptasNaor} for the graph coloring problem. In particular it was
proved there that for a given number of colors there exists a critical connectivity value
which separates the colorable from the un-colorable phases.
The connection among the two approaches has emerged
in recent times also within the algorithmic setting. A method has been developed to study 
graph colorability \cite{braunstein2003polynomial} 
based on ideas from the physics of disordered systems, in particular on the replica
symmetry breaking scheme introduced within the mean field theory
of spin glasses \cite{mezard2001bethe}.

In this paper we obtain a full control of a region in the high temperature phase of the model:
computing the free energy and identifying a phase transition at a critical $\beta^{crit}(c,q)$.
Our main result is the following.
\begin{theorem}
For a given number of colors $q>1 $ and a Poissonian Erd{\H o}s-R\'enyi random graph 
of parameter $c>0$, define the "annealed'' pressure
\begin{equation}
\label{eq:annealedpressure}
\mathscr{P}(\beta,c)\, =\, \ln q + \frac{c}{2}\, \ln\left(1 - \frac{1-e^{-\beta}}{q}\right)\, .
\end{equation}
Define moreover 
\begin{gather*}
\beta_{RS}^{\text{loc}}(c,q)\, 
=\, 
\begin{cases}
\infty & \text{ for $c\leq c_{RS}^{\text{loc}}(q)$,}\\[2pt]
\displaystyle -\ln\left(1-\frac{q}{1+\sqrt{c}}\right)  & \text{ otherwise,}
\end{cases}\\[5pt]
\beta_{\text{ent}}(c,q)\, 
=\, 
\begin{cases}
\infty & \text{ for $c\leq c_{\text{ent}}(q)$,}\\[2pt]
\displaystyle \inf\left\{\beta\, :\, \ln(q) + \frac{c}{2}\, \ln\left(1 - \frac{1-e^{-\beta}}{q}\right) 
\leq \frac{\beta c}{2} \cdot \frac{e^{-\beta}}{q-1+e^{-\beta}}\right\} & \text{ otherwise,}
\end{cases}\\[5pt]
\beta_1(c,q)\,
=\,  \begin{cases} \beta_{RS}^{\text{loc}}(c,2) & \text{ for $q=2$,}\\[2pt]
\infty & \text{ for $q>2$ and $c\leq c_1(q)$,}\\[2pt]
\displaystyle{-\ln\left(1 - \frac{q}{q-1+\sqrt{c/\left(2q\ln(q)\right)}}\right)}\,  & \text{ otherwise,}
\end{cases}
\end{gather*} 
where
$$
c_{RS}^{\text{loc}}(q):=(q-1)^2\, ,\qquad
c_{\text{ent}}(q):=\frac{2\ln(q)}{|\ln(1-q^{-1})|}\quad \text{ and }\quad
c_1(q):=2 q \ln(q)\, .
$$
There is no phase transition at any finite temperature if $c<c_1(q)$.
For $c > \min\{c_{RS}^{\text{loc}}(q),c_{\text{ent}}(q)\}$,
there exist a phase transition at the critical value $\beta^{crit}(c,q)$ with
$$
\beta_1(c,q)\le\beta^{crit}(c,q)\le\min\{\beta_{RS}^{loc}(c,q),\beta_{\text{ent}}(c,q)\}.
$$
The quenched pressure is equal to $\mathscr{P}(\beta,c)$ for $\beta \leq \beta_1(c,q)$ and 
it is different (stricly less) for
$\beta> \min\{\beta_{RS}^{loc}(c,q),\beta_{\text{ent}}(c,q)\}$.
\end{theorem}

The proof of the theorem will be a combination of different results obtained in the following
sections. The method we follow combines ideas developed within the rigorous
theory of spin glasses \cite{aizenman2003extended} with second moment bounds
 \cite{AchlioptasNaor}.
A full treatment of the ferromagnetic Ising case
has been given in \cite{DemMon10} for locally tree-like random graphs
and extended in \cite{DommersGiardina-van-der-Hofstad}.
The techniques used there are heavily based on the use of ferromagnetic
Griffiths-Kelly-Sherman and Griffiths-Hurst-Sherman inequalities and
do not apply to our case.
For the antiferromagnetic model we introduce here an interpolation scheme and prove
its monotonicity (see \cite{bayati2010combinatorial} for an alternative interpolation scheme 
in the Bernoulli case).


The paper is organized  as follows. The model is defined in section 2
and the interpolation is introduced in section 3.  The extended variational
principle that applies to our case is formulated and studied in section 4.
Derrida-Ruelle like trial states are described in section 5 and then
used in section 6 to obtain replica symmetry breaking bounds.
Section 7 develops the constrained second moment computation,
along the lines of  the previous zero-temperature computations.
Details of the proofs and explicit computations are included in the
Appendices and make the paper self-contained.

\section{The model}
\label{model}

We start by considering the general set-up for the $q$-state Potts model for $q\in\N$.
We use the notation from combinatorics
$$
[q]\, =\, \{1,2,\dots,q\}\, .
$$
We consider a set of $N$ vertices, such that for each $i=1,\dots,N$,
there is a spin variable
$\s_i \in [q]$.
Given a subset $\mathcal{S} \subseteq \R$ let $\M_N(\mathcal{S})$ be the set of all $N\times N$ matrices
with entries in $\mathcal{S}$.
The $q$-state Potts Hamiltonian is
$$
H_N\, :\, [q]^N \times \M_N(\R) \to \R\, ,\qquad
H_N(\s,\bJ)\, =\, \sum_{i,j=1}^{N} J_{ij} \delta(\s_i,\s_j)\, ,
$$
where $\delta(r,s)$ is the Kronecker delta for $r,s \in [q]$:
1 if $r=s$ and 0 otherwise.

For a general $\bJ \in \M_N(\R)$, we may define the usual thermodynamic quantities.
\begin{gather}
\text{partition function:}\qquad 
Z_N(\bJ) \;=\; \sum_{\sigma\in[q]^{N}} e^{-H_N(\sigma,\bJ)}\, ,\\
\label{eq:BGomega}
\text{Boltzmann-Gibbs measure:}\qquad 
\s \in [q]^N \mapsto \omega_{N,\bJ}(\s)\, =\, \frac{e^{-H_N(\sigma,\bJ)}}{Z_N(\bJ)}\, ,\\[3pt]
\text{Boltzmann-Gibbs expectation:}\qquad
(f:[q]^N \to \R) \mapsto
\langle f\rangle_{N,\bJ}\, 
=\, \sum_{\s \in [q]^N} f(\s) \omega_{N,\bJ}(\s)\, .
\end{gather}
For now, we have absorbed the inverse temperature $\beta$ into the coupling matrix $\bJ$. But later we will make it explicit.

For a general $R \in \N$, and a function of $R$ replicas, $F : ([q]^N)^R \to \R$, we use the same notation
$$
\langle F\rangle_{N,\bJ}\, 
=\, \sum_{\s^{(1)},\dots,\s^{(R)} \in [q]^N} F(\s^{(1)},\dots,\s^{(R)}) \prod_{r=1}^{R} \omega_{N,\bJ}(\s^{(r)})\, .
$$
Often we gather all $R$ replicas as 
$$
\Sigma_R\, =\, (\sigma^{(1)},\dots,\sigma^{(R)})\, .
$$ 
The set of all $\Sigma_R$'s will be denoted $[q]^{N \times R} = \{(\s^{(1)},\dots,\s^{(R)})\, :\, \s^{(1)},\dots,\s^{(R)} \in [q]^N\}$.

The finite-volume approximation to the pressure is
$$
\p_N(\bJ)\, =\, \frac{1}{N}\, \ln Z_N(\bJ)\, .
$$
If $\bJ$ is random then $\p_N(\bJ)$ is, too.
But we will more frequently use a different notation for the quenched pressure, where we take the expectation of $\p_N(\bJ)$ over the disorder distribution of $\bJ$.

\subsection{The disorder distribution}
Let $\N_0$ denote $\{0,1,2,\dots\}$.
For each $\lambda\geq 0$, let $\pi_{\lambda} : \N_0 \to [0,1]$ denote the Poisson-$\lambda$ mass function
\begin{equation}
\label{eq:Poisson}
\pi_{\lambda}(0)\, =\, e^{-\lambda}\, ,\qquad
\pi_{\lambda}(k)\, =\, e^{-\lambda}\, \frac{\lambda^k}{k!}\ \text { for $k\in\{1,2,\dots\}$.}
\end{equation}
A key feature for the interpolation method of Franz and Leone \cite{FranzLeone} for Poisson couplings, generalizing the Guerra-Toninelli interpolation for Gaussians
can be called Poisson summation by parts:
\begin{equation}
\label{eq:PoissonSBP}
\frac{d}{d\lambda} \pi_{\lambda}(k)\, =\, \pi_{\lambda}(k-1) - \pi_{\lambda}(k)\quad \Rightarrow\quad
\frac{d}{d\lambda} E^{\pi_{\lambda}}[f]\, =\, \E^{\pi_{\lambda}}[f(\cdot+1) - f(\cdot)]\, .
\end{equation}

Given $\bc \in \M_N([0,\infty))$, define the measure $\P_{N,\bc}$ on $\M_N(\N_0)$
as
$$
\P_{N,\bc}(A)\, 
=\, \sum_{\bJ \in \M_N(\N_0)} \boldsymbol{1}_A(\bJ)
\prod_{i,j=1}^{N} \pi_{c_{ij}/(2N)}(J_{ij}) \, .
$$
Let $\E_{N,\bc}$ denote the expectation with respect to the probability measure $\P_{N,\bc}$.
It is frequently useful to use the notation
\begin{equation}
\label{eq:doubleE}
\big\langle\hspace{-3pt}\big\langle \cdots \big\rangle\hspace{-3pt}\big\rangle_{N,\beta,\bc}\, =\, \E_{N,\bc}\left[\big\langle \cdots \big\rangle_{N,\beta \bJ}\right]\, .
\end{equation}
Given $\beta \in [0,\infty)$ and $\bc \in \M_N([0,\infty))$, let us define the quenched pressure
\begin{equation}
\label{eq:DefQuench}
p_N(\beta,\bc)\, =\, \E_{N,\bc}\left[\p_N(\beta \bJ)\right]\, .
\end{equation}
Given $c \in [0,\infty)$, let $\P_{N,c}$ denote the measure $\P_{N,\bc}$ for the matrix $\bc$ such that
$c_{ij} = c$ for all $i,j\in\{1,\dots,N\}$,
and similarly let $\E_{N,c}$ and $\p_N(\beta,c)$ denote $\E_{N,\bc}$ and $p_N(\beta,\bc)$ for this choice of $\bc$.
Similarly,
let 
$$
\big\langle \hspace{-2pt} \big\langle \cdots \big\rangle \hspace{-2pt}\big\rangle_{N,\beta,c}
$$
denote $\big\langle \hspace{-2pt} \big\langle \cdots \big\rangle \hspace{-2pt}\big\rangle_{N,\beta,\bc}$
for this special choice of $\bc$.
Occasionally it is necessary to explicitly denote the dependence of $p_N(\beta,c)$ on $q$
in which case we write $p_N(\beta,c,q)$.

\section{Interpolation}

\label{sec:interpolation}

In the present model we use interpolation  to prove existence of the thermodynamic limit of the quenched pressure
(see \cite{bayati2010combinatorial} for the antiferomagnetic model with Bernoulli dilution, not the Poissonian
case we consider here).
The method of interpolation is a well-known tool for disordered mean-field models of statistical mechanics
(see \cite{FranzLeone} for the diluted spin-glass and \cite{GT1} for the Sherrington-Kirkpatrcik model).
%
%
\begin{lemma}
\label{lem:interpo}
Given any differentiable curve $t \mapsto \bc(t)$ in $\M_N([0,\infty))$, 
$$
\frac{d}{dt}\,  p_N(\beta,\bc(t))\, 
=\, \frac{1}{2N^2}\, \sum_{i,j=1}^{N} \frac{dc_{ij}}{dt}\, \E_{N,\bc(t)}\left[\ln\left(1 - \left(1-e^{-\beta}\right) \left\langle \delta(\s_i,\s_j)\right\rangle_{N,\beta\bJ}\right)\right]\, .
$$
\end{lemma}
\begin{proof}
This follows from a well-known calculation which we include for the benefit of the reader.
By (\ref{eq:PoissonSBP}) and the chain rule,
$$
\frac{d}{dt}\, \E_{N,\bc(t)}[\p_N(\beta\bJ)]\, 
=\,  \frac{1}{2N}\, \sum_{i,j=1}^{N}  \frac{dc_{ij}}{dt}\, \E_{N,\bc(t)}\left[\p_N(\beta\bJ)\Big|_{J_{ij} \to J_{ij}+1} - \p_N(\beta\bJ)\right]\, ,
$$
But since $\p_N(\beta \bJ) = N^{-1} \ln Z_N(\beta \bJ)$,
$$
\p_N(\beta\bJ)\Big|_{J_{ij} \to J_{ij}+1} - \p_N(\beta\bJ)\, 
=\, \frac{1}{N}\, \ln \left(\frac{Z_N(\beta\bJ)\Big|_{J_{ij} \to J_{ij}+1}}{Z_N(\beta\bJ)}\right)\,
=\, \frac{1}{N}\, \ln \left\langle e^{-\beta \delta(\s_i,\s_j)}\right\rangle_{N,\beta\bJ}\, .
$$
Using the fact that $e^{-\beta \delta(\s_i,\s_j)} = 1 - \left(1-e^{-\beta}\right) \delta(\s_i,\s_j)$, this gives the desired result.
\end{proof}

The first corollary is existence of the thermodynamic limit. 
\begin{corollary}
\label{cor:superadd}
For $\beta,c \geq 0$, and any $N_1,N_2 \in \N$,
\begin{equation}
\label{eq:pNsuperadd}
p_{N_1+N_2}(\beta,c)\, \geq\, \frac{N_1}{N_1+N_2}\, p_{N_1}(\beta,c) + \frac{N_2}{N_1+N_2}\, p_{N_2}(\beta,c)\, .
\end{equation}
\end{corollary}

This states that the sequence $(p_N(\beta,c))_{N\in\N}$ is superadditive.
We will prove this in Section \ref{sec:ProofsInterpolation}. 
Let us now state an inequality for superadditive sequences.
\begin{lemma}
\label{lem:superadd}
If $(x_N)_{N \in \N}$ satisfies $(M+N) x_{M+N} \geq M x_M + N x_N$ for all $M,N \in \N$,
then
$$
\liminf_{N \to \infty} x_N\, =\, \limsup_{N \to \infty} x_N\, =\, \sup_{N \in \N} x_N\, =\, \limsup_{N \to \infty} \liminf_{M \to \infty} \frac{(M+N) x_{M+N} - M x_M}{N}\, .
$$
\end{lemma}
The first part of this lemma is a result due to Fekete.
The last equality follows from an argument in \cite{aizenman2003extended}.
It will be useful later.
We will review the proof in Section \ref{sec:ProofsInterpolation}.

Let us now introduce an important function, which is called the {\it annealed} pressure
\begin{equation}
\label{eq:annealedpressure}
\mathscr{P}(\beta,c)\, =\, \ln q + \frac{c}{2}\, \ln\left(1 - \frac{1-e^{-\beta}}{q}\right)\, .
\end{equation}
We call it annealed with a slightly different meaning than the spin glass case,
as it will be clear in the following.
This function provides an upper bound for the quenched pressure $p_N(\beta,c)$ for every $N \in \N$, as we will show next.
In order to state the precise result,
recall that $\Sigma_R = (\sigma^{(1)},\dots,\sigma^{(R)}) \in [q]^{N \times R}$ is a notation gathering $R$ replicas.
Given $\Sigma_R$, let us define
the $R$-replica empirical measure on $[q]^R$:
\begin{equation}
\label{eq:qRreplica}
\rho_{\Sigma_R}(s)\, =\, \frac{1}{N}\, \sum_{i=1}^{N} \prod_{r=1}^{R}\delta(\s^{(r)}_i,s_r)\quad \text{ for }\quad
s = (s_1,\dots,s_R) \in [q]^R\, .
\end{equation}

\begin{theorem} 
\label{thm:Annealed}
For every $\beta,c \geq 0$,
\begin{equation}
\label{eq:SumRule}
\mathscr{P}(\beta,c) - p_N(\beta,c)\, =\, \frac{1}{2} \sum_{R=0}^{\infty} \frac{(1-e^{-\beta})^R}{R}\, \sum_{s \in [q]^R} 
\int_0^c \left\langle\hspace{-4pt}\left\langle \left(\rho_{\Sigma_R}(s) - q^{-R}\right)^2 \right\rangle \hspace{-4pt}\right\rangle_{N\beta,c'}\, dc'\, .
\end{equation}
\end{theorem}
As a particular implication, note that $p_N(\beta,c) \leq \mathscr{P}(\beta,c)$ for all $N$.
Along with Corollary \ref{cor:superadd} and Lemma \ref{lem:superadd}, this implies that the thermodynamic pressure
exists as a finite limit
$$
p(\beta,c)\, \stackrel{\text{def}}{:=}\, \lim_{N \to \infty} p_N(\beta,c)\, ,
$$
and it satisfies $p(\beta,c) \leq \mathscr{P}(\beta,c)$.

\begin{remark}
When it is necessary to explicitly denote the dependence on $q$ we will write $p(\beta,c,q)$ and $\mathscr{P}(\beta,c,q)$.
\end{remark}

The explicit formula for $\mathscr{P}(\beta,c)-p(\beta,c)$ is relevant when trying to determine the {\em annealed region}: the parameter space for $(\beta,c) \in [0,\infty)^2$ such that the 
inequality is saturated,
$p(\beta,c) = \mathscr{P}(\beta,c)$.

The final application of interpolation is the analogue of Guerra replica symmetry breaking bounds \cite{guerra}.
We introduce this in the next section in order to give the full definition of the random spin structure which aids in understanding those inequalities. See  \cite{PanchenkoTalagrand}, \cite{BovierKlimovsky}, \cite{talagrand2006parisi} \cite{ArguinChatterjee} for
similar results.
%
%
We include proofs for the benefit of the reader in the Appendix.

Before ending this section let us note another elementary corollary which is useful in the next section.
\begin{corollary}
\label{cor:LipGeneralp}
Suppose that $\bc^{(1)}$ and $\bc^{(2)}$ are both in $\M_N([0,\infty))$.
Then, 
$$
\left|p_N(\beta,\bc^{(2)}) 
- p_N(\beta,\bc^{(1)})\right|\,
\leq\, \frac{\beta}{2N^2}\, \sum_{i,j=1}^{N} \left|c_{ij}^{(2)} - c_{ij}^{(1)}\right|\, .
$$
In particular, for two different numbers $c_1, c_2 \geq 0$, we have 
$$
\left|p_N(\beta,c_2) 
- p_N(\beta,c_2)\right|\, \leq\, \frac{1}{2}\, \beta|c_2-c_1|\, .
$$
\end{corollary}

\section{Extended Variational Principle}
\label{sec:EVP}
\label{Extended Variational Principle}

We follow here the method introduced by
Aizenman, Sims and Starr in \cite{aizenman2003extended,ASS2}.
We start this section by defining a discrete random spin structure.
The definition comes from the physicists' cavity step,
as defined by Franz and Leone \cite{FranzLeone}.
%
%
Recall that with the usual topology $[0,1]$ is compact.
Let $[0,1]^{\N}$ be the set of all 
$\bxi = (\xi_1,\xi_2,\dots)$ such that each $\xi_i \in [0,1]$.
With the product topology, $[0,1]^{\N}$ is also compact,
and metrizable.
For example, a metric compatible with the product topology is 
$d(\bxi,\boldsymbol{\zeta}) = \sum_{n=1}^{\infty} 2^{-n} |\xi_n-\zeta_n|$.

Let $\Delta$ denote the subset consisting of those 
$\bxi \in [0,1]^{\N}$
satisfying the additional conditions
$$
\xi_1\, \geq\, \xi_2\, \geq\,  \dots\quad \text{ and } \quad
\xi_1 + \xi_2 + \dots\, \leq\, 1\, .
$$
This is a closed set, hence also compact.

We also define $\Delta_1$ to be the subset consisting of those $\bxi \in \Delta$
such that $\sum_{n=1}^{\infty}\xi_n = 1$.
This is not a closed set, but it is a Borel set: $\bigcap_{m=1}^{\infty} \bigcup_{n=1}^{\infty} \{\sum_{k=1}^{n} \xi_k \geq 1-m^{-1}\}$.

Let $q^{\N}$ refer to the set of $\tau = (\tau_1,\tau_2,\dots)$ with each $\tau_n \in [q]$.
Let $q^{\N\times \N}$ refer to the set of all $\mathcal{T} = (\tau^{(1)},\tau^{(2)},\dots)$ with each $\tau^{(n)} \in q^{\N}$.
With the product topology, $q^{\N \times \N}$ is also compact and metrizable.

Finally, let $S_{\infty}$ denote the set of all bijections $\pi : \N \to \N$ such that $\{n\, :\, \pi(n)\neq n\}$ is finite.
Given $\tau \in q^{\N}$ and $\pi \in S_{\infty}$, we defined $\tau \circ \pi \in q^{\N}$
such that $(\tau \circ \pi)_n = \tau_{\pi(n)}$.

\begin{definition}
\label{def:Kingman}
(a)
Let $\mathscr{M}$ denote the set of all Borel probability measures on $\Delta \times q^{\N \times \N}$.\\
(b) Let $\mathscr{S}$ denote the subset of all $\mathcal{L} \in \mathscr{M}$ satisfying additional
hypotheses:
\begin{itemize}
\item[(i)] $\mathcal{L}(\{(\bxi,\mathcal{T})\, :\, \bxi \in \Delta_1\})\, =\, 1$,
\item[(ii)]
For any $\pi \in S_{\infty}$, and any Borel subset $A \subseteq \Delta \times q^{\N\times \N}$
$$
\mathcal{L}(\{(\bxi,\mathcal{T})\, :\, (\bxi,(\tau^{(1)}\circ \pi,\tau^{(2)}\circ \pi,\dots)) \in A\})\, 
=\, \mathcal{L}(A)\, .
$$
\end{itemize}
\end{definition}
The set of all discrete random spin structures is $\mathscr{S}$.
In Section \ref{sec:Cavity} we will discuss a generalization of this definition which represents a compactification.
But for now, we define the cavity field functions.

Given $k$, let $\bI \in \{1,\dots,N\}^k$ denote $(I_1,\dots,I_k)$.
Let us denote the union
$$
\mathcal{I}_N\, \stackrel{\text{def}}{:=}\, \bigcup_{k=0}^{\infty} \{1,\dots,N\}^k\, .
$$
Given $\bI \in \mathcal{I}_N$, we define $|\bI|$ to be that integer $k \in \N_0$
such that $\bI \in \{1,\dots,N\}^k$.
Note that for $k=0$, we just denote $\bI$ to be a placeholder $\emptyset$.
We define a probability measure on this space
$$
\widetilde{\P}_{N,c}(A)\, =\, \sum_{k=0}^{\infty} \frac{\pi_{cN}(k)}{N^k} \sum_{\bI \in \{1,\dots,N\}^k} \boldsymbol{1}_A(\bI)\, .
$$
Let $\widetilde{\E}_{N,c}$ be the associated expectation.
We also define a Hamiltonian
$$
\widetilde{H}_N : \mathcal{I}_N \times [q]^{\N} \times [q]^N \to \R\, ,\qquad
\widetilde{H}_N(\bI,\tau,\sigma)\, =\, \sum_{i=1}^{|\bI|} \delta(\tau_i,\sigma_{I_i})\, .
$$
For $\bI=\emptyset$, we have $|\bI|=0$ and the empty sum is interpreted as zero.
With all of this set-up, we define the ``interaction'' term of the  cavity field function to be
\begin{equation}
\label{eq:GN1}
G_N^{(1)}(\beta,c,\mathcal{L})\,
=\, \int_{\Delta_1 \times [q]^{\N\times \N}} \widetilde{\E}_{N,c}\left[
\frac{1}{N} \ln \sum_{\alpha=1}^{\infty} \xi_{\alpha} \sum_{\s \in [q]^N} \exp\left(-\beta \widetilde{H}_N(\bI,\tau^{(\alpha)},\sigma)\right)\right]\,  d\mathcal{L}(\bxi,\mathcal{T})\, .
\end{equation}
The ``reaction'' (or self-energy) term is 
\begin{equation}
\label{eq:GN2}
G^{(2)}_{N}(\beta,c,\mathcal{L})\,
=\, \int_{\Delta_1 \times [q]^{\N\times \N}} \sum_{K=0}^{\infty} \frac{\pi_{cN/2}(K)}{N}\, \ln\left( \sum_{\alpha=1}^{\infty} \xi_{\alpha} \exp\left(-\beta \sum_{k=1}^{K} \delta(\tau^{(\alpha)}_{2k-1},\tau^{(\alpha)}_{2k})\right)\right)\, 
d\mathcal{L}(\bxi,\mathcal{T})\, .
\end{equation}
The analogue of Guerra's replica symmetry breaking bounds (\cite{guerra}) are the following.
\begin{theorem}
\label{thm:RSBbd}
For any $\beta,c\geq 0$ and $N \in \N$, and for any $\mathcal{L} \in \mathscr{S}$,
$$
p_N(\beta,c)\, \leq\, G_N^{(1)}(\beta,c,\mathcal{L}) - G_N^{(2)}(\beta,c,\mathcal{L})\, .
$$
\end{theorem}
We will prove this in the Appendix. 

In the next section we will use random spin structures coming from the Poisson-Dirichlet, Derrida-Ruelle random probability cascade.
But first, we try to motivate the present formulation by indicating how to obtain opposite bounds using the Boltzmann-Gibbs random spin structures.
It is these opposite bounds that are most closely related to the physicists' original perspective on the cavity step \cite{MezardParisiVirasoro}.


\begin{theorem}
\label{thm:EVP}
For any $\beta,c \geq 0$,
\begin{equation}
\label{eq:EVP}
\lim_{N \to \infty} p_N(\beta,c)\, =\, \lim_{N \to \infty} \inf_{\mathcal{L} \in \mathscr{S}} G_{N}(\beta,c,\mathcal{L})\, .
\end{equation}
\end{theorem}
We do not use this theorem for any further applications in this paper.
But its proof helps to motivate the definition of the $N$-step cavity field functionals.
The proof of the theorem will be given in the next subsection.

\subsection{Boltzmann-Gibbs Spin Structures}

Now we construct an example of a discrete random spin structure,
which we will call $\mathcal{L}_{N,\beta,c}$.
This is derived from the Boltzmann-Gibbs distribution, itself.

Let $\bJ$ be distributed according to $\P_{N,c}$.
Let $\mathcal{N} = q^N$.
Let $\s^{(1)},\dots,\s^{(\mathcal{N})}$ be any enumeration of $[q]^N$.
For $\alpha \in \{1,\dots,\mathcal{N}\}$, let
$$
\xi_{\alpha}\, =\, \omega_{N,\beta\bJ}(\s^{(\alpha)})\, .
$$
For $\alpha > \mathcal{N}$, let $\xi_{\alpha}=0$ and let $\s^{(\alpha)} \in [q]^N$
be any configuration. The choice of $\s^{(\alpha)}$ does not matter since $\xi_{\alpha}=0$.
Let $I_1,I_2,\dots \in \{1,\dots,N\}$ be i.i.d., uniformly distributed on $\{1,\dots,N\}$,
independent of $\bJ$.
For each $\alpha \in \N$, let $\tau^{(\alpha)} \in [q]^{\N}$ be 
$$
\tau^{(\alpha)}_k\, =\, \s^{(\alpha)}_{I_k}\quad \text{ for $k \in \N$.}
$$
The measure $\mathcal{L}_{N,\beta,c}$ describes the marginal distribution
of $(\bxi,\mathcal{T})$.

The key identity for proving Theorem \ref{thm:EVP} is as follows.
\begin{lemma}
\label{lem:EstimateEVP}
We have the identities
$$
G_N^{(2)}(\beta,c,\mathcal{L}_{M,\beta,c'})\,
=\, \frac{M}{N}\left(p_{M}\left(\beta,c'+\frac{cN}{M}\right) - p_M(\beta,c')\right)\, ,
$$
and
$$
G_{N,c}^{(1)}(\beta,\mathcal{L}_{M,\beta,c'})\,
=\, \frac{M+N}{N}\, p_{M+N}\left(\beta,\hat{\boldsymbol{c}}^{(M,N)}\right) - \frac{M}{N}\, p_M(\beta,c')\, ,
$$
where $p_N(\beta,\boldsymbol{c})$ for a general matrix $\bc \in \M_N([0,\infty))$ was defined in (\ref{eq:DefQuench})
and the matrix
$\hat{\boldsymbol{c}}^{(M,N)} \in \M_{M+N}([0,\infty))$ is defined as 
$$
\hat{c}^{(M,N)}_{i,j}\, 
=\, \begin{cases}
c' (1+\frac{N}{M}) & \text{ if $i,j \leq M$,}\\
c (1+\frac{N}{M}) & \text { if $i \leq M$, $j>M$ or if $j\leq M$, $i>M$,}\\
0 & \text { if $i,j > M$.}
\end{cases}
$$
\end{lemma}

We will prove this lemma in the Appendix.
It follows from the definitions
and 
infinite divisibility of the Poisson process.
Infinite divisibility is merely the mathematical condition related to the fact that
the Poisson random
variables admit interpolation.

The physicists' cavity step amounts to considering a very large system in equilibrium.
We will say that the size is $M$.
Then the physicists consider removing a smaller number of spins, say $N\leq M$,
which creates a cavity in the system.
But mathematically one can instead consider adding $N$ spins. (In other words
the added spins are a cavity in a system of size $M+N$.)

This has two effects. Firstly, each of the $N$ spins interacts with all the $M$ spins
in a mean-field way, i.e., in a way that represents the underlying symmetry of the model, called exchangeability. To leading order this is represented by $G_N^{(1)}$.
The simplification occurs because the leadin order terms in the interaction are linear.
In other words, for each of the $N$ spins it is as if it feels a random external magnetic
field, with the distribution of this magnetic field determined by the $M$ spins in ``equilibrium,''
and some extra random couplings.

The second effect is a reaction or self-energy term for the $M$ spins.
This is because, being a mean-field model, the parameter of the model $c$
is actually being scaled by the reciprocal of the system size.
So changing the system size amounts to a renormalization of the connectivity
from $c$ to $c(1+\frac{N}{M})$.
To leading order, the self-energy for the $M$ spins is represented by $G_N^{(2)}$
which actually does not depend on the spins $\s_1,\dots,\s_N$ at all, only the spins
in the $M$ ``equilibrium'' system.

There are other terms in the Hamiltonian,
amounting to interactions with two or more spins among the $N$ subsystem. 
But taking all these terms together still only gives a lower-order effect  which may be neglected in the thermodynamic limit.
In essence, Lemma \ref{lem:EstimateEVP} is just a calculation to show
that we have correctly interpreted the physicists' cavity step.

\begin{proofof}{\bf Proof of Theorem \ref{thm:EVP}:}
The upper bounds of Theorem \ref{thm:RSBbd} imply that
$$
p(\beta,c)\, =\, \lim_{N \to \infty} p_N(\beta,c)\, \leq\, \liminf_{N \to \infty} \inf_{\mathcal{L} \in \mathscr{S}} \left(G_{N}^{(1)}(\beta,c,\mathcal{L}) - G_{N}^{(2)}(\beta,c,\mathcal{L})\right)\, .
$$
All we need to do is to establish the opposite bound,
\begin{equation}
\label{ineq:opp}
p(\beta,c)\, \geq\, \limsup_{N \to \infty} \inf_{\mathcal{L} \in \mathscr{S}} \left(G_{N}^{(1)}(\beta,c,\mathcal{L}) - G_{N}^{(2)}(\beta,c,\mathcal{L})\right)\, .
\end{equation}
From Corollary \ref{cor:superadd} and Lemma \ref{lem:superadd},
we know that
\begin{equation}
\label{eq:pFekete}
p(\beta,c)\, =\, \limsup_{N \to \infty} \liminf_{M \to \infty} \left(\frac{M+N}{N}\, p_{M+N}(\beta,c) - 
\frac{M}{N}\, p_M(\beta,c)\right)\, .
\end{equation}
But by Lemma \ref{lem:EstimateEVP}, we know that
\begin{equation}
\label{eq:pEstimate}
G_{N,c}^{(1)}(\beta,\mathcal{L}_{M,\beta,c'}) - G_N^{(2)}(\beta,c,\mathcal{L}_{M,\beta,c'})\,
=\, \frac{M+N}{N}\, p_{M+N}\left(\beta,\hat{\boldsymbol{c}}^{(M,N)}\right) - 
\frac{M}{N}\, p_{M}\left(\beta,c'+\frac{cN}{M}\right)\, ,
\end{equation}
where 
$$
\hat{c}^{(M,N)}_{i,j}\, 
=\, \begin{cases}
c' (1+\frac{N}{M}) & \text{ if $i,j \leq M$,}\\
c (1+\frac{N}{M}) & \text { if $i \leq M$, $j>M$ or if $j\leq M$, $i>M$,}\\
0 & \text { if $i,j > M$.}
\end{cases}
$$
Choosing $c' = c / (1+\frac{N}{M})$, we see that
$$
\frac{M}{N}\, p_{M}\left(\beta,c'+\frac{cN}{M}\right) - 
\frac{M}{N}\, p_M(\beta,c)\,
=\, \frac{M}{N}\, \left(p_{M}\left(\beta,\frac{c}{1+\frac{N}{M}}+\frac{cN}{M}\right) - 
p_M(\beta,c)\right)\, .
$$
Using the bounds from Corollary \ref{cor:LipGeneralp}, this implies
\begin{equation}
\label{ineq:G2}
\left|\frac{M}{N}\, p_{M}\left(\beta,c'+\frac{cN}{M}\right) - 
\frac{M}{N}\, p_M(\beta,c)\right|\,
\leq\, \frac{\beta cN}{2(M+N)}\, .
\end{equation}
Similarly, using the fact that $p_{M+ N}(\beta,c) = p_{M+N}(\beta,\bc)$ for the matrix $\bc$
with $c_{ij}  =c$ for all $i,j$, we see that (choosing $c'$ as before)
$$
\left|\frac{M+N}{N}\, p_{M+N}\left(\beta,\hat{\boldsymbol{c}}^{(M,N)}\right) - 
\frac{M+N}{N}\, p_{M+N}(\beta,c)\right|\, 
\leq\, \frac{\beta c N}{M+N}\, ,
$$
using the matrix-version bound from Corollary \ref{cor:LipGeneralp}.
So, putting this together with (\ref{eq:pEstimate}) and (\ref{ineq:G2}), we have
$$
\left|
G_{N,c}^{(1)}(\beta,\mathcal{L}_{M,\beta,c'}) - G_N^{(2)}(\beta,c,\mathcal{L}_{M,\beta,c'})\,
- \left(\frac{M+N}{N}\, p_{M+N}(\beta,c) - 
\frac{M}{N}\, p_M(\beta,c)\right)\right|\, \leq\, \frac{3\beta c N}{2(M+N)}\, .
$$
Since this bound vanishes in the limit $M \to \infty$, before $N$ goes to $\infty$,
and since the
Boltzmann-Gibbs spin structure is just one particular choice of a random spin structure,
so that the true infimum is no greater than this,
we see that (\ref{eq:pFekete}) does imply (\ref{ineq:opp}), as desired.
\end{proofof}

\section{Derrida-Ruelle Construction}
\label{Derrida-Ruelle Random Probability Cascade}


Theorem \ref{thm:EVP} shows that the cavity functional
$$
G_N^{(1)}(\beta,c,\mathcal{L}) - G_N^{(2)}(\beta,c,\mathcal{L})
$$
needs to be minimized over discrete random spin structures $\mathcal{L} \in \mathscr{S}$. The optimal choice of the
measure has been conjectured to be described by a construction based on the Derrida-Ruelle random probability cascade
\cite{Derrida, Ruelle}.
The results we obtain in this section provide a rigorous proof to some physicist's results
obtained with heuristic methods in \cite{KZ, ZdeborovaKrzakala}.

%

\subsection{The Ultrametric space}

The Derrida-Ruelle probability cascade construction is based on a rooted tree with finitely many levels.
Let us define $\mathcal{T}_0 = \{\emptyset\}$ where $\emptyset$ will denote a single vertex at the root level.
For $\ell \in \N$, let $\mathcal{T}_\ell = \N^\ell$.
So a typical element of $\mathcal{T}_\ell$ is $\ba = (\alpha_1,\dots,\alpha_\ell)$ with $\alpha_1,\dots,\alpha_\ell \in \N$.
Let us denote this as $\ba^\ell = (\alpha_1,\dots,\alpha_{\ell})$ in order to explicitly denote the depth $\ell$.

Given $\ba^{\ell} = (\alpha_1,\dots,\alpha_{\ell})$ in $\mathcal{T}_{\ell}$, let us define $\ba^{\ell}_{\restriction k} = (\alpha_1,\dots,\alpha_k)$
in $\mathcal{T}_k$ for each $k=1,\dots,\ell$.
Then, given $L \in \N$, we define a tree of depth $L$ as 
$\mathscr{T}_L$ which has vertex set
$$
\mathscr{T}_L \, \stackrel{\text{def}}{:=}\, \mathcal{T}_0 \sqcup \mathcal{T}_1 \sqcup \cdots \sqcup \mathcal{T}_L\, ,
$$
and such that the mother of each $\ba^1 = (\alpha_1)$ in $\mathcal{T}_1$ is the root $\emptyset \in \mathcal{T}_0$
and the mother of each $\ba^{\ell} \in \mathcal{T}_{\ell}$ for $\ell=2,\dots,L$ is $\ba^{\ell}_{\restriction k}$.
As usual for trees, two vertices are connected if and only if one is the mother of the other one, called the daughter.

The leaf set of a tree is the set of all vertices which have no daughters.
So this is $\mathcal{T}_L$ for $\mathscr{T}_L$.
Next we define a family of random probability distributions
on the leaf set.

Let $V_L$ denote the set of all $L$-tuples
$\bm^{(L)} = (m_1,\dots,m_L)$ satisfying
$$
0\, <\, m_1\, <\, \dots\, <\, m_L\, <\, 1\, .
$$
For consistency, we define $V_0 = \{\emptyset\}$.
For each $L \in \N_0$
and each $\bm^{(L)} \in V_L$, we will define a probability distribution
giving rise to random variables $\widehat{\xi}_{\bm^{(L)}}(\ba^{(L)})$ 
for each $\ba^{(L)} \in \mathcal{T}_L$,
which are nonnegative and such that
$$
\sum_{\ba^{(L)} \in \mathcal{T}_L} \widehat{\xi}_{\bm^{(L)}}(\ba^{(L)})\, =\, 1\, ,
$$
almost surely, for each choice of $\bm^{(L)}$.
We use the hat to denote normalization, since we will construct
the probability measure $\widehat{\xi}_{\bm^{(L)}}(\ba^{(L)})$ by normalizing
an almost surely normalizable measure $\xi_{\bm^{(L)}}(\ba^{(L)})$.

We can define this inductively as follows.
We start by defining $\hat{\xi}_{\emptyset}$ to be the unique  (hence non-random) probability measure on $\mathcal{T}_0$:
$\widehat{\xi}_{\emptyset}(\emptyset)=1$.

\subsection{The Poisson-Dirichlet Derrida-Ruelle distributions}

To extend to the definition of $\xi_{\bm^{(L)}}$ to $L \in \N$
and $\bm^{(L)} \in V_L$,
we will first 
quickly review the definition of a general Poisson point process.
This is because our construction uses Poisson-Dirichlet distributions,
based on Poisson point processes.
But also, for certain proofs, the general definition of a Poisson point process will be useful.

Suppose that $\mathcal{X}$ is a locally compact metric space.
Suppose that $\Lambda$ is a locally finite Borel measure on $\mathcal{X}$, meaning
that for any compact set $K \subset \mathcal{X}$, we have $\Lambda(K)<\infty$.
Given this, one may define the Poisson process with intensity measure $\Lambda$
to be $\Xi$, a random point process, meaning that $\Xi$ is a 
random $\N_0$-valued measure.
Given $n \in \N$ and given disjoint compact sets $K_1,\dots,K_n \subseteq \mathcal{X}$,
we have the marginal distribution
$$
\P(\Xi(K_1)=k_1,\dots,\Xi(K_n)=k_n)\, =\, \prod_{i=1}^{n} \pi_{\Lambda(K_i)}(k_i)\, ,
$$
for each choice of $k_1,\dots,k_n \in \N_0$.
We remind the reader that the Poisson distribution was defined in (\ref{eq:Poisson}).

Due to infinite divisibility this is a consistent definition in the sense of the Kolmogorov consistency
conditions. It also leads to the alternative description in terms of the moment generating functional.
Suppose that $f : \mathcal{X} \to [0,\infty)$ is any Borel measurable function.
Then
\begin{equation}
\label{eq:MGFal}
\E\left[\exp\left(-\int_{\mathcal{X}} f(x) d\Xi(x)\right)\right]\,
=\, \exp\left(-\int_{\mathcal{X}} (1-e^{-f(x)})\, d\Lambda(x)\right)\, .
\end{equation}
This identity being true for all nonnegative, Borel measurable functions is equivalent
to the consistent family of marginal distributions described above.
This general framework will be useful shortly.
Among many good reviews of Poisson processes, Ruelle's paper on Derrida's REM and GREM
is an exemplary reference \cite{Ruelle}.

Now we define the random measure $\widehat{\xi}_{\bm^{(1)}}$ on $\mathcal{T}_1 = \N^1$ for each
choice of $\bm^{(1)} = (m_1)$ with $m_1 \in (0,1)$.
Let us denote $m_1$ as just $m$ for this case, $L=1$.
Let $\Lambda_m$ be the following locally finite measure on $\mathcal{X} = (0,\infty)$,
$$
d\Lambda_m(x)\, =\, mx^{-m-1}\, dx\, .
$$
Let $\Xi$ be the 
An example of an easy calculation with (\ref{eq:MGFal}) is the following:
\begin{lemma}
\label{lem:Laplace}
For any $p>m$ and any $\lambda>0$,
$$
\E\left[\exp\left(-\lambda \int_0^{\infty} x^p\, d\Xi(x)\right)\right]\, =\, \exp\left(-\lambda^{m/p} \int_0^{\infty} x^{-m/p} e^{-x}\, dx\right)\, .
$$
\end{lemma}
This will be proved in Section \ref{sec:PD}. This implies that $\int_0^\infty x\, d\Xi(x)$ is in $(0,\infty)$, almost surely.
(Taking $\lambda$ to $0$ we see that the probability to be $\infty$ is zero,
and taking $\lambda \to \infty$, we see that the probability to be $0$ is zero.)
In turn this implies that almost surely we can identify points 
$$
\xi_1\, \geq\, \xi_2\, \geq\, \dots\, >\, 0\, ,
$$ 
such that $\Xi(A) = \sum_{n=1}^{\infty} \boldsymbol{1}_A(\xi_n)$ for every Borel set $A \subseteq (0,\infty)$, and $\sum_{n=1}^{\infty} \xi_n$ is in $(0,\infty)$, almost surely.
A key property is the following stability property, whose proof may be found in the paper 
\cite{AizenmanRuzmaikina}:
\begin{theorem}
\label{thm:stability}
Suppose that $X_1,X_2,\dots$ are i.i.d., positive random multipliers, a.s., independent of $\xi_1,\xi_2,\dots$,
and such that $\E[X_i^m] < \infty$.
Then the random point process $A \mapsto \sum_{n=1}^{\infty} \boldsymbol{1}_A(X_n \xi_n)$
is equal in distribution to the random point process $A \mapsto \sum_{n=1}^{\infty} \boldsymbol{1}_A(c \xi_n)$
for $c = (\E[X_i^m])^{1/m}$.
\end{theorem}
We will give a few hints of the proof in the Appendix.

Then we define $\widehat{\xi}_{(m)}((\alpha))$ for all $\alpha \in \N$ as follows:
$$
\widehat{\xi}_{(m)}((\alpha))\, =\, \frac{\xi_{\alpha}}{\xi_1+\xi_2+\dots}\quad \text{ for $\alpha\in \N$.}
$$
The distribution of this random discrete probability measure is called the Poisson-Dirichlet distribution $\operatorname{PD}(m,0)$.
It is one branch of the two-parameter Poisson-Dirichlet distributions 
(see \cite{Pitman}).

Note that we have now defined $\widehat{\xi}_{\bm^{(1)}}(\ba^{(1)})$ for all $\bm^{(1)} \in V_1$ and $\ba^{(1)} \in \mathcal{T}_1$,
satisfying the desired conditions, almost surely.
Now we define $\widehat{\xi}_{\bm^{(L)}}$ for all $L\geq 1$ and $\bm^{(L)} \in V_L$, inductively.
We have defined it above for $L=1$ and $\bm^{(1)} = (m) \in V_1$.
Assuming $L$ is in $\{2,3,\dots\}$ and that we have defined the measure for all depths less than $L$,
we treat the case of depth $L$ as follows.

First, using the induction hypothesis we may assume the existence of random variables 
$$
\widehat{\xi}_{\bm^{(L)}_{\restriction L-1}}(\ba^{(L-1)}) \quad \text{ for all $\ba^{(L-1)} \in \mathcal{T}_{L-1}$,}
$$
where $\bm^{(L)}_{\restriction L-1}$ is defined as the restriction to the first $L-1$ coordinates of $\bm^{(L)} = (m_1,\dots,m_L)$.
Then, independently of that, for all $\ba^{(L-1)} \in \mathcal{T}_{L-1}$, let us take $\Xi^{(\ba^{(L-1)})}$ to be a Poisson
point process with intensity $\Lambda_{m_L}$, such that all the Poisson point processes are independent
for different choices of $\ba^{(L-1)} \in \mathcal{T}_{L-1}$.
Each one may be written as 
$$
\Xi^{(\ba^{(L-1)})}(A)\, =\, \sum_{n=1}^{\infty} \boldsymbol{1}_A(\xi^{(\ba^{(L-1)})}_n)
\quad \text { for all measurable $A \subseteq (0,\infty)$,}
$$
for some random numbers $\xi^{(\ba^{(L-1)})}_1 \geq \xi^{(\ba^{(L-1)})}_2 \geq \cdots > 0$.
The following is a corollary of Lemma \ref{lem:Laplace} and Theorem \ref{thm:stability},
which we will prove this in the Apeendix. 
\begin{corollary}
\label{cor:Z}
Assuming $0<m_1<\dots<m_L$ then the nonnegative random variable 
$$
Z(\bm^{(L)})\, \stackrel{\text{def}}{:=}\, \sum_{\ba^{(L)} \in \mathcal{T}_L} \widehat{\xi}_{\bm^{(L)}_{\restriction L-1}}(\ba^{(L)}_{\restriction L-1})
\cdot \xi^{(\ba^{(L)}_{\restriction L-1})}_{\alpha_L}
$$
satisfies $0<Z(\bm^{(L)})<\infty$, almost surely.
\end{corollary}

Then we complete the induction step by defining
$$
\widehat{\xi}_{\bm^{(L)}}(\ba^{(L)})\, \stackrel{\text{def}}{:=}\, 
\frac{1}{Z(\bm^{(L)})}\, \widehat{\xi}_{\bm^{(L)}_{\restriction L-1}}(\ba^{(L)}_{\restriction L-1})
\cdot \xi^{(\ba^{(L)}_{\restriction L-1})}_{\alpha_L}\quad 
\text{ for all $\ba \in \mathcal{T}_L$,}
$$
which is well-defined and normalized, almost surely.

Next one constructs a probability measure on spins, indexed by leaves of the tree. 

\subsection{The 
measures on measures construction
}

\label{subsec:FL}

Let $\mathcal{M}_1$ denote the set of all probability measures on $[q]$.
This is a finite-dimensional simplex. 
Using the topology of weak-convergence on probability measures
this simplex has its usual topology.
In particular it is compact and metrizable.

Let $\mathcal{M}_2$ denote the set of all Borel probability measures on $\mathcal{M}_1$.
Then, with the topology of weak convergence, this is also compact and metrizable.
Indeed, the set of Borel measures on a compact, metrizable set is always itself
compact and metrizable when equipped with the topology of weak-convergence.

Therefore, inductively, for all $\ell \in \N$, we let $\mathcal{M}_{\ell+1}$ denote the
set of all Borel probability measures on $\mathcal{M}_\ell$, equipped with the topology
of weak-convergence.
We denote a measure in $\mathcal{M}_{\ell+1}$ as $\mu^{(\ell+1)}$.
But we note that the standard notation for its differential is somewhat
cumbersome $d\mu^{(\ell+1)}(\mu^{(\ell)})$.

Now let $\mu^{(L)}$ denote any measure in $\mathcal{M}_L$.
This is our input.
In order to initialize the induction step, we change notation slightly,
$$
\mu^{(L)}_{\emptyset}\, \stackrel{\text{def}}{:=}\, \mu^{(L)}\, .
$$
For each $\ba^{(1)} = (\alpha_1) \in \mathcal{T}_1$,
let $\mu^{(L-1)}_{\ba^{(1)}}$ be a random element of $\mathcal{M}_{L-1}$,
distributed according to $\mu^{(L)}_{\emptyset}$, and such that they are all independent
for different choices of $\ba^{(1)} \in \mathcal{T}_1$.

Continue inductively.
For $\ell=2,\dots,L-1$ let $\mathcal{F}_{\ell-1}$ denote the $\sigma$-algebra generated by
all the random variable that were constructed at the previous level,
$$
\mu^{(L-k+1)}_{\ba^{(k-1)}}\quad \text{ for all $k\leq \ell$ and $\ba^{(k-1)} \in \mathcal{T}_{k-1}$.}
$$
We construct $\mu^{(L-\ell)}_{\ba^{(\ell)}}$ for all $\ba^{(\ell)} \in \mathcal{T}_{\ell}$
as follows.

Conditionally, given $\mathcal{F}_{\ell-1}$,
let $\mu^{(L-\ell)}_{\ba^{(\ell)}}$ be a random element of $\mathcal{M}_{L-\ell}$,
distributed according to $\mu^{(L-\ell+1)}_{\ba^{(\ell)}_{\restriction \ell-1}}$.
More precisely choose these random variables $\mu^{(L-\ell)}_{\ba^{(\ell)}}$,
for each $\ba^{(\ell)} \in \mathcal{T}_{\ell}$, such that they are all conditionally independent,
conditional on $\mathcal{F}_{\ell}$.

Finally, given all this, for each $\ba^{(L)} \in \mathcal{T}_{L}$ and each $i \in \N$, let 
$\tau_i(\ba^{(L)})$ be distributed according to $\mu^{(1)}_{\ba^{(L)}_{\restriction L-1}}$,
such that they are all conditionally independent, conditional on $\mathcal{F}_1$.
Let us define
$$
\tau(\ba^{(L)}) = (\tau_1(\ba^{(L)}),\tau_2(\ba^{(L)}),\dots) \in [q]^{\N}\, ,
$$
for each $\ba^{(L)} \in \mathcal{T}_L$.
Then we may consider the pairs consisting  of $(\widehat{\xi}_{\bm^{(L)}}(\ba^{(L)}))_{\ba^{(L)} \in \mathcal{T}_L}$
and $(\tau(\ba^{(\ell)}))_{\ba^{(\ell)} \in \mathcal{T}_L}$.
Note that $\mathcal{T}_L$ is countable.
We denote the distribution of such pairs as $\mathcal{L}_{\bm^{(L)},\mu^{(L)}}$.
Then since the set of possible $\alpha$, here replaced by $\ba^{(L)} \in \mathcal{T}_L$,
is countable, this is an example of a discrete random spin structure in 
$\mathscr{S}$ as in Definition \ref{def:Kingman}.

\begin{remark}
The necessity to introduce the measure on measure structure comes from the fact that,
unlike in gaussian spin glass where the infinitely divisible distribution allows a continuous
parametrization of ansatz, here the lack of the property of infinite divisibility forces the 
introduction of discrete iteration ansatz in the optimization procedure.
\end{remark}

\section{
``Replica Symmetry Breaking'' bounds}

We obtain here rigorous bounds as a consequence of Theorem \ref{thm:RSBbd}.

\subsection{One level trees and the annealed bounds}

The simplest case to consider is $L=1$.
Then $\bm^{(1)} = (m_1)$ for some $m_1 \in (0,1)$.
For this case, we choose to rewrite $m_1$ as just $m \in (0,1)$,
so that $\bm^{(1)} = (m)$.
In this case we have a Poisson-Dirichlet distribution which according to our previous notation is
$$
\widehat{\xi}_{(m)}((1))\, ,\ \widehat{\xi}_{(m)}((2))\, ,\ \dots\, .
$$
We prefer to work directly with the Poisson point process $\xi_1 \geq \xi_2 \geq \dots > 0$,
with intensity measure $\Lambda_m$, defining
$$
Z\, =\, \sum_{n=1}^\infty \xi_n\, ,
$$
which is almost surely in $(0,\infty)$. Then $\widehat{\xi}_{(m)}((\alpha))$ is equal to
$\xi_{\alpha}/Z$.
It will turn out that the effect of the normalization $Z$ will cancel in the formula for
$$
G_N^{(1)}(\beta,c,\mathcal{L}_{\bm^{(1)},\mu^{(1)}}) - G_N^{(2)}(\beta,c,\mathcal{L}_{\bm^{(1)},\mu^{(1)}})\, .
$$
But by using the Poisson point process directly, instead of the normalized Poisson-Dirichlet process, we may
may appeal to Theorem \ref{thm:stability} to help in the calculations of 
(\ref{eq:GN1}) and (\ref{eq:GN2}).

Let us also refer to $\tau^{(\alpha)} = (\tau^{(\alpha)}_1,\tau^{(\alpha)}_2,\dots)$,
which are i.i.d., distributed according to $\mu^{(1)}_{\emptyset}=\mu^{(1)}$
for some non-random measure $\mu^{(1)} \in \mathcal{M}_1$. 
For each $\alpha \in \N$, and $\bI \in \mathcal{I}_N$, let us define
$$
X_{\alpha}(\bI)\, =\, \sum_{\sigma \in [q]^N}\exp\left(-\beta \widetilde{H}_N(\bI,\tau^{(\alpha)},\sigma)\right)\, .
$$
Then, conditioning on $\bI$, these are i.i.d., random variables in $\alpha$.
In other words, the random variables $\tau^{(\alpha)}_{1},\tau^{(\alpha)}_2,\dots$,
are all i.i.d., for different $\alpha$'s.
Therefore,
the resulting marginal distribution of the $X_{\alpha}(\bI)$'s are i.i.d (for each fixed $\bI \in \mathcal{I}_N$).
Then Theorem \ref{thm:stability} implies that
$$
\sum_{\alpha=1}^{\infty} \xi_{\alpha} \sum_{\sigma \in [q]^N}\exp\left(-\beta \widetilde{H}_N(\bI,\tau^{(\alpha)},\sigma)\right)\,
=\, \sum_{\alpha=1}^{\infty} \xi_{\alpha} X_{\alpha}(\bI)\, 
\stackrel{\mathcal{D}}{=}\, \E[X_{\alpha}(\bI)^m\, |\, \bI]^{1/m} \sum_{\alpha=1}^{\infty} \xi_{\alpha}\, ,
$$
where we indicate equality in distribution by $\mathcal{D}$.

Note that the sum of the $\xi_{\alpha}$'s is $Z$ which is the normalization.
So, since we are taking the logarithm, 
$$
G_N^{(1)}(\beta,c,\mathcal{L}_{\bm^{(1)},\mu^{(1)}})\,
=\, \frac{1}{mN}\, \widetilde{\E}_{N,c}\left[ \ln \E[X_{\alpha}(\bI)^m\, |\, \bI] \right]\, ,
$$
where the inner conditional expectation is over the $\mu^{(0)}_{\alpha}$'s and the $\tau^{(\alpha)}$'s,
but not $\bI$.
Similarly, we obtain
$$
G_N^{(2)}(\beta,c,\mathcal{L}_{\bm^{(1)},\mu^{(1)}})\,
=\, \frac{1}{mN}\, \sum_{K=0}^{\infty} \pi_{cN/2}(K) \ln \E[Y_{\alpha}(K)^m]\, ,
$$
where we define
$$
Y_{\alpha}(K)\, =\, \exp\left(-\beta \sum_{k=1}^{K} \delta(\tau_{2k-1}^{(\alpha)},\tau_{2k}^{(\alpha)})\right)\, .
$$
A very easy warm-up is the limiting case $m \uparrow 1$.
Note that this limit is not a discrete spin structure in $\mathscr{S}$.
In the Appendix
we will mention a compactification.
But this is not necessary, here. 
For each $m$ we have the upper bound
$$
p(\beta,c)\, 
\leq\, \frac{1}{mN}\, \widetilde{\E}_{N,c}\left[ \ln \E[X_{\alpha}(\bI)^m\, |\, \bI] \right]
- \frac{1}{mN}\, \sum_{K=0}^{\infty} \pi_{cN/2}(K) \ln \E[Y_{\alpha}(K)^m]\, .
$$
The right hand side is continuous in $m$.
Therefore, taking the limit as $m \to 1$, we still have the upper bound
\begin{equation}
\label{eq:pAnn2}
p_N(\beta,c)\,
\leq\, \frac{1}{N}\, \widetilde{\E}_{N,c}\left[ \ln \E[X_{\alpha}(\bI)\, |\, \bI] \right]
- \frac{1}{N}\, \sum_{K=0}^{\infty} \pi_{cN/2}(K) \ln \E[Y_{\alpha}(K)]\, .
\end{equation}
Finally, to make the bound even easier we may take $\mu^{(1)} \in \mathcal{M}_1$
to be the uniform measure on $[q]$.
In other words,  the $\tau^{(\alpha)}_i$'s are i.i.d., random, uniformly
distributed on $[q]$.
For this simplified case,
\begin{multline*}
\E[Y_{\alpha}(K)]\,
=\, \E\left[\exp\left(-\beta \sum_{k=1}^{K} \delta(\tau_{2k-1}^{(\alpha)},\tau_{2k}^{(\alpha)})\right)\right]\,
=\, \left[1 - \frac{1-e^{-\beta}}{q}\right]^K\, ,\\
\Rightarrow\quad
\frac{1}{N}\, \sum_{K=0}^{\infty} \pi_{cN/2}(K) \ln \E[Y_{\alpha}(K)]\,
=\, \frac{c}{2}\, \ln\left(1-\frac{1-e^{-\beta}}{q}\right)\, .
\end{multline*}
Similarly,
\begin{multline*}
\E[X_{\alpha}(\bI)\, |\, \bI]\,
=\, 
\sum_{\sigma \in [q]^N}
\E\left[\exp\left(-\beta \sum_{i=1}^{|\bI|} \delta(\tau^{(\alpha)}_i,\sigma_{I_i})\right)\, \bigg|\, \bI\right]\, 
=\, q^N \left[1 - \frac{1-e^{-\beta}}{q}\right]^{|\bI|}\, ,\\
\Rightarrow\quad
\frac{1}{N}\, \widetilde{\E}_{N,c}\left[ \ln \E[X_{\alpha}(\bI)\, |\, \bI] \right]\,
=\, \ln(q) +
c \ln\left(1-\frac{1-e^{-\beta}}{q}\right)\, .
\end{multline*}
Therefore, combining this with (\ref{eq:pAnn2}),
we obtain the bound
$$
p_N(\beta,c)\, \leq\, \ln(q) + \frac{c}{2}\, \ln\left(1-\frac{1-e^{-\beta}}{q}\right)\, ,
$$
which re-derives the annealed upper bound (\ref{eq:SumRule}) in Theorem \ref{thm:Annealed},
without the sum-rule correction.
In fact, one can include the correction term also in the analogue of Guerra's upper bound
in Theorem \ref{thm:RSBbd}.
But we did not do this here, because we do not have any method to control the error term.

The ansatz we have taken here is not  the  general case of the so-called ``replica symmetric'' ansatz.
We will explain that in the next section: the difference is that there should be two steps, $m_1<m_2$
and then one takes the limit $m_1\downarrow 0$, $m_2\uparrow 1$.
Instead we just have 1 level, with $m\uparrow 1$. 
So this is a specialized ansatz, which one could call the ``trivial replica symmetric ansatz.''
Moreover, we chose the most basic choice for $\mu^{(1)} \in \mathcal{M}_1$.
So we could call this the ``trivial, symmetric replica symmetric ansatz.''
Next we will consider a more refined upper bound (see 
\cite{KZ,ZdeborovaKrzakala}).
There existes a local instability point within the replica symmetric ansatz, where another replica symmetric 
trial state gives a lower bound than the trivial, symmetric replica symmetric ansatz.

\label{subsec:OneLevel}

\subsection{Two level trees and the replica symmetric ansatz}

Recall that $\mathcal{F}_L \subseteq \mathcal{F}_{L-1} \subseteq \dots \subseteq \mathcal{F}_1$, defined in Section \ref{subsec:FL}
is a reversed filtration.
\begin{lemma} For $L\geq 0$ and $\bm^{(L)} \in V_L$ and $\mu^{(L)} \in \mathcal{M}_L$,
define two sequences of random variables:
$X_{\alpha}^{(L)}(\bI) = X_{\alpha}(\bI)$ and $Y_{\alpha}^{(L)}(K) = Y_{\alpha}(K)$
as defined in Section \ref{subsec:OneLevel}, and for $\ell=1,\dots,L-1$,
\begin{gather*}
X_{\alpha}^{(\ell)}(\bI)\, =\, \E\left[X_{\alpha}^{(\ell+1)}(\bI)^{m_{\ell+1}}\, |\, \mathcal{F}_{L-\ell}\vee\sigma(\bI)\right]^{1/m_{\ell+1}}\quad 
\text{ and}\\
Y_{\alpha}^{(\ell)}(K)\, =\, \E\left[Y_{\alpha}^{(\ell+1)}(K)^{m_{\ell+1}}\, |\, \mathcal{F}_{L-\ell}\right]^{1/m_{\ell+1}}\, .
\end{gather*}
Then the cavity field functionals are calculated at the final step of the backward iteration
\begin{gather*}
G_N^{(1)}(\beta,c,\mathcal{L}_{\bm^{(L)},\mu^{(L)}})\,
=\, \frac{1}{m_1 N}\, \widetilde{\E}_{N,c}\left[\ln  \E\left[X_{\alpha}^{(1)}(\bI)^{m_1}\, |\, \bI\right]\right]\quad 
\text{ and}\\
G_N^{(2)}(\beta,c,\mathcal{L}_{\bm^{(L)},\mu^{(L)}})\,
=\, \frac{1}{m_1 N}\,  \sum_{K=0}^{\infty} \pi_{cN/2}(K) \ln \E[Y_{\alpha}^{(1)}(K)^{m_1}]\, .
\end{gather*}
\label{lem:tree}
\end{lemma}
This lemma is proved just like Corollary \ref{cor:Z}, proved in the Appendix.

The replica symmetric ansatz is obtained by taking a $L=2$ level tree and then taking the limit $m_1\downarrow 0$, $m_2\uparrow 1$.
In order to derive the relevant limit note that if one takes $m_1 \downarrow 0$ in Lemma \ref{lem:tree} then the last step of the backward
iteration is  
\begin{equation}
\label{eq:G0}
\begin{split}
G_N^{(1)}(\beta,c,\mathcal{L}_{\bm^{(L)},\mu^{(L)}})\,
=\, \frac{1}{N}\, \widetilde{\E}_{N,c}  \E\left[\ln\left(X_{\alpha}^{(1)}(\bI)\right)\right]\quad 
\text{ and}\\
G_N^{(2)}(\beta,c,\mathcal{L}_{\bm^{(L)},\mu^{(L)}})\,
=\, \frac{1}{N}\,  \sum_{K=0}^{\infty} \pi_{cN/2}(K) \E\left[\ln\left(Y_{\alpha}^{(1)}(K)\right)\right]\, ,
\end{split}
\end{equation}
which is a standard calculation based on the fact that $\lim_{m \downarrow 0} m^{-1} \ln \E[X^m] = \E[\ln(X)]$
for random variables $X$ satisfying mild conditions to allow the application of the dominated convergence theorem.
A sufficient condition is that both $X^m$ and $\ln(X)$ are integrable, which is satisfied in the formulas above.
If we let $L=2$ and take the limit $m_2\uparrow 1$ then in addition to (\ref{eq:G0}) we have 
\begin{equation}
\label{eq:RS1}
X_{\alpha}^{(1)}(\bI)\, =\, \E[X_{\alpha}(\bI)\, |\, \mathcal{F}_1\vee\sigma(\bI)]\quad \text{ and } \quad
Y_{\alpha}^{(1)}(K)\, =\, \E[Y_{\alpha}(K)\, |\, \mathcal{F}_1]\, .
\end{equation}
In principle, equations (\ref{eq:G0}) and (\ref{eq:RS1}) determine the replica symmetric ansatz,
once one makes a choice for $\mu^{(2)} \in \mathcal{M}_2$, a non-random measure on measures.

For $s\in [q]$ and $t$ satisfying $-(q-1)^{-1} \leq t\leq 1$, denote a measure $\mu^{(1)}_{s,t} \in \mathcal{M}_1$, i.e., a measure on $[q]$, such that
$$
\mu^{(1)}_{s,t}(\{r\})\, =\, t \delta(r,s) + \frac{1-t}{q}\, \delta(r,s))\, .
$$
For a chosen $t$, let $\mu^{(2)}_t \in \mathcal{M}_2$ be the measure on measures, such that
$$
\mu^{(2)}_{t}(\{\mu^{(1)}_{s,t}\})\, =\, \frac{1}{q}\quad \text{ for each $s \in [q]$.}
$$
In other words, we may consider $\mu^{(1)}_{\alpha}$ in the following way for each $\alpha \in \N$.
Let $S_{\alpha}$ be chosen uniformly at random in $[q]$, such that all the $S_{\alpha}$'s are independent.
Then let $\mu^{(1)}_{\alpha} = \mu^{(1)}_{S_{\alpha},t}$.
This has the right distribution.

Note that this is a very specific choice; it is not general.
But it is a choice which makes the following analysis simpler.
Also note that taking the special value $t=0$ then $\mu^{(1)}_{s,0}$ is uniform on $[q]$, not depending
on $s$.
Therefore at this point all the $\mu^{(1)}_{\alpha}$'s are uniform on $[q]$, and are therefore non-random.
From this it is apparent that taking $t=0$ recovers the ``trivial, symmetric replica symmetric'' ansatz
of the last section which led to the annealed upper bound.
We now want to use this set-up to derive the following result 
\begin{corollary}
\label{cor:KZ}
Suppose $q>1$. If 
$$
c\, >\, c_{RS}^{\text{loc}}(q)\, \stackrel{\text{def}}{:=}\, (q-1)^2\quad \text{ and } \quad
\beta\, >\, \beta_{RS}^{\text{loc}}(c,q)\, \stackrel{\text{def}}{:=}\, -\ln\left(1-\frac{q}{1+\sqrt{c}}\right)\, ,
$$
then $p(\beta,c) < \mathscr{P}(\beta,c)$. The quenched pressure is strictly less than the annealed pressure.
Moreover, within the replica symmetric ansatz this is due to a local instability, commonly associated to a second order
phase transition.
\end{corollary}
%

\begin{proof}
The proof is a corollary of  Theorem \ref{thm:RSBbd}.
The values we obtain for $\mu^{(2)}=\mu^{(2)}_p$, written above, are
\begin{align*}
\lim_{\substack{m_1 \downarrow 0 \\ m_2 \uparrow 1}} G_N^{(1)}(\beta,c,\mathcal{L}_{\bm^{(2)},\mu^{(2)}_t})\,
&=\, \ln(q) + c \ln\left(1 - \frac{1-e^{-\beta}}{q}\right) + g^{(1)}(\beta,c,q,t)\quad \text{ and}\\
\lim_{\substack{m_1 \downarrow 0 \\ m_2 \uparrow 1}} G_N^{(1)}(\beta,c,\mathcal{L}_{\bm^{(2)},\mu^{(2)}_t})\,
&=\, -\frac{c}{2}\, \ln\left(1 - \frac{1-e^{-\beta}}{q}\right) + g^{(2)}(\beta,c,q,t)\, ,
\end{align*}
where
\begin{align*} 
g^{(1)}(\beta,c,q,t)\,
&=\, \sum_{k=0}^{\infty} \pi_{c}(k) \sum_{\tau_1,\dots,\tau_k \in [q]} q^{-k} \ln \Bigg(\sum_{s\in[q]} q^{-1} \prod_{i=1}^{k} (1-xt(q \delta(\tau_i,s)-1))\Bigg)\quad \text{ and}\\
g^{(2)}(\beta,c,q,t)\, 
&=\, \frac{c}{2q}\, \left((q-1) \ln(1+xt^2) + \ln(1 - (q-1)xt^2)\right)\, , 
\end{align*}
and in both expressions
$$
x\, =\, x(\beta,q)\, \stackrel{\text{def}}{:=}\, \frac{1-e^{-\beta}}{q-1+e^{-\beta}}\, .
$$
These calculations are straightforward given the definitions above, but require some steps to prove.
Therefore, we relegate the derivation to section \ref{app:Calculations}.
For now we use these formulas to finish the argument for the proof of the present corollary.

Since $g^{(1)}$ and $g^{(2)}$ are both zero when $p=0$, we note that the difference between the upper bound
obtained by the replica symmetric ansatz with $p\neq 0$ and the annealed bound is
$$
g^{(1)}(\beta,c,q,t) - g^{(2)}(\beta,c,q,t)\, .
$$
If one can prove that for any $t \in [-(q-1)^{-1},1]$ this difference is strictly negative, then that
will establish an upper bound for $p(\beta,c)$ which is strictly less than $\mathscr{P}(\beta,c)$.
It suffices to do a perturbative argument for $t$ close to zero, and establish that the leading order
term is negative.
It is straightforward to Taylor expand these two functions in $t$. We claim that
\begin{align}
g^{(1)}(\beta,c,q,t)\, &=\, -\frac{1}{4}\, (q-1) c^2 x^4 t^4 + O(t^6)\quad \text{ and }\\
g^{(2)}(\beta,c,q,t)\, &=\, -\frac{1}{4}\, (q-1) c x^2 t^4 + O(t^6)\quad \text{ as $t\to 0$.}
\end{align}
This is another calculation which we prefer to derive carefully in section
\ref{app:Calculations}.
From this one can see that there is an instability of the ``trivial, symmetric replica symmetric'' ansatz, i.e., the leading order term as $t\to 0$ is negative
meaning that an asymmetric replica symmetric ansatz gives an even lower trial for the minimizer, if
$$
\frac{1}{4}\, (q-1) \left[c^2 x^4 - c x^2\right]\, >\, 0\quad \Leftrightarrow\quad
c x^2\, >\, 1\, .
$$
But recalling the definition of $x=x(\beta,q)$ above, this means
$$
\frac{1-e^{-\beta}}{q-1+e^{-\beta}}\, >\, \frac{1}{\sqrt{c}}\, ,
$$
and this leads to the conditions stated as the hypothesis of the corollary.
\end{proof}

Note that this result shows a local instability within the replica symmetric ansatz.
One can also consider a 1-level replica symmetry breaking ansatz, which amounts to taking $L=3$,
and taking the limit $m_1 \downarrow 0$ and $m_3 \uparrow 1$, but keeping $m_2$ strictly
between $0$ and $1$ as a generic point, representing the height of the middle level.
%

So far we have proved annealed upper bounds for all $\beta$ and $c$,
$$
p(\beta,c,q)\, \leq\, \mathscr{P}(\beta,c,q)\, ,
$$
from Theorem \ref{thm:Annealed}.
We have also proved that for a certain regime we must have strict inequality: $p(\beta,c,q) < \mathscr{P}(\beta,c,q)$
when the hypotheses of Corollary \ref{cor:KZ} are satisfied, which are conditions on the triple $(\beta,c,q)$.
Next we prove that in a certain high-temperature or low-connectivity regime, the annealed pressurre is correct,
so that one has equality rather than strict inequality.

\label{subsec:KZ}


\section{Constrained Second Moment Method}
\label{sec:AN}
The $\beta \to \infty$ limit of the $q$-state Potts model is the $q$-coloring problem. Namely, if for any edge two vertices have
the same color then the $\beta \to \infty$ limit gives the entire coloring probability zero.
So, if there are any proper colorings of the full graph with $q$-colors, then the $\beta \to \infty$ limit of the Boltzmann-Gibbs measure
should be the uniform measure on the set of all proper colorings.

The problem was studied in \cite{AchlioptasNaor} where, using a variant of the second moment method, the critical q for each c was identified within an interval. The error bound, when translated to relative errors, are vanishingly small in the $q\to\infty$ limit. The physics interpretation of their results is that even at $\beta=\infty$, if c is sufficiently small with $q$ fixed, the model is in the high-temperature, low-connectivity region. 
Here we tackle instead the positive temperature regime and we state our result in two separate cases.

\begin{theorem}
\label{thm:AN}
(1)
If $q=2$ then for $\beta \leq \beta_*(c,2) = \beta_{\text{RS}}^{\text{loc}}(c,2)$, we have $p(\beta,c,q) = \mathscr{P}(\beta,c,q)$.\\
(2)
If $q>2$ then there is a $\beta_*(c,q)$ such that for $\beta \leq \beta_*(c,q)$, we have $p(\beta,c,q) = \mathscr{P}(\beta,c,q)$.
Moreover we have lower bounds on $\beta_*(c,q)$:
$$
\beta_*(c,q)\, \geq\, \beta_1(c,q)\, =\, -\ln\left(1 - \frac{q}{q-1+\sqrt{c/[2q\ln(q)]}}\right)\, .
$$
\end{theorem}

A few comments are in order. Firstly, \cite{KZ} it was conjectured that the critical temperature of this model
is the same as for the model with extra randomness: for every edge present, one has a an independent,
uniform random permutation $\pi$ on $[q]$, such that instead of the term $\delta(\s_i,\s_j)$ in the Hamiltonian
one has $\delta(\s_i,\pi(\s_j))$.
This is one of several possible extensions of the Viana-Bray model for $q>2$.
But for $q=2$ it is equivalent to the Viana-Bray model.
If one takes the conjecture in \cite{KZ} for granted, then for $q=2$ the critical temperature
would be deduced from work on the Viana-Bray model by Guerra and Toninelli \cite{GueTon04}.
In particular, our result (1) does confirm this picture.
Note that $\beta_*(c,2)$ is the correct value since we know $p(\beta,c,2)\neq \mathscr{P}(\beta,c,2)$ for $\beta>\beta_*(c,2)$
by Corollary \ref{cor:KZ}.

For $q>2$ the second moment method introduced in \cite{AchlioptasNaor}  leads to an optimization
which presumably has a trivial solution for a larger region than one can prove.
Therefore we do not claim that $\beta_1(c,q)$ is sharp.
On the other hand, solving the equation for the relationship between $c$ and $q$ in the $\beta \to \infty$ 
limit, we do recover the zero-temperature limit of the pressure: for a fixed $q$, we do have 
$p(\infty,c,q) = \mathscr{P}(\infty,c,q)$ as long as $c\leq 2 q \ln(q)$.
%

\subsection{Bounds from entropy positivity}
\label{subsec:ANSK}

Let us define the random entropy density at finite volumes as
$$
\mathfrak{s}_N(\bJ)\, \stackrel{\text{def}}{:=}\, -\frac{1}{N}\, \sum_{\sigma \in [q]^N} \omega_{N,\bJ}(\s) \ln \omega_{N,\bJ}(\s)\, ,
$$
for each $\bJ \in \M_N(\R)$.
In particular, note that for finite $N$, if $\bJ \in \M_N([0,\infty))$ admits a zero energy ground state, then
$$
\lim_{\beta \to \infty} \mathfrak{s}_N(\beta \bJ)\, =\, \frac{1}{N} \ln\left(|\{\s \in [q]^N\, :\, H_N(\s,\bJ)=0\}|\right)\, ,
$$
as one desires. In particular it is nonnegative.
In fact, whether or not $\bJ$ admits a zero-energy ground state, the entropy is always nonnegative
because for each $\s \in [q]^N$, the probability $\omega_{N,\bJ}(\s)$ is at most 1 because this is a discrete
probability measure.
Therefore, $-\ln(\omega_{N,\bJ}(\s))\geq 0$.
Also, it is easy to see that
$$
-\beta\, \frac{\partial}{\partial \beta} \mathfrak{p}_N(\beta\bJ)\,
=\, -\frac{\beta}{N} \cdot \frac{\partial}{\partial \beta} \ln Z_N(\beta \bJ)\,
=\, \frac{1}{N}\, \sum_{\s \in [q]^N} \beta H_N(\s,\beta \bJ) \omega_{N,\beta \bJ}(\s)\,
=\, \mathfrak{s}_N(\beta \bJ) + \mathfrak{p}_N(\beta J)\, .
$$
So, since $\mathfrak{p}_N(\beta \bJ)$ is a convex function of $\beta$, we see that 
$$
\mathfrak{s}_N(\beta \bJ)\, =\, - \beta\, \frac{\partial}{\partial \beta} \mathfrak{p}_N(\beta \bJ) - \mathfrak{p}_N(\beta \bJ)\,
\leq\, \frac{\beta}{\delta}\left[\mathfrak{p}_N([\beta-\delta] \bJ) - \mathfrak{p}_N(\beta \bJ)\right] - \mathfrak{p}_N(\beta,\bJ)\, ,
$$
for any $\delta>0$.
This leads to the following conclusion:

\begin{corollary}
\label{cor:ANSK}
For $q>1$, 
let $\beta^*(c,q)$ be the infimum of the set $\{\beta \geq 0\, :\, p(\beta,c,q) \neq \mathscr{P}(\beta,c,q)\}$.
Then 
$$
\beta^*(c,q)\, \leq\, \inf\left\{\beta\, :\, \ln(q) + \frac{c}{2}\, \ln\left(1 - \frac{1-e^{-\beta}}{q}\right) 
\leq \frac{\beta c}{2} \cdot \frac{e^{-\beta}}{q-1+e^{-\beta}}\right\}\, .
$$
\end{corollary}

Since we proved in Theorem \ref{thm:AN} that $\beta^*(c,q)>0$ for all $c$ and $q$, this result implies a phase
transition, i.e., existence of a critical temperature, as long as $c>c^*(q)\stackrel{\text{def}}{:=} 2 \ln(q)/|\ln(1-q^{-1})|$.
In other words, $p(\beta,c,q)$ cannot be analytic beyond this point, because of the ``identity theorem'' from complex analysis:
it is identically equal to $\mathscr{P}(\beta,c,q)$ for a positive interval $\beta<\beta^*(c,q)$,
and $\mathscr{P}(\beta,c,q)$ is analytic.
Therefore, it would have to equal $\mathscr{P}(\beta,c,q)$ identically, unless there is a phase transition in the sense of a 
point of non-analyticity.
It is interesting to compare this to Corollary \ref{cor:KZ}.
If $c>\min\{c^*(q),c^{\text{loc}}_{\text{RS}}(q)\}$, then there is a phase transition.
For $q=2,3$ the smaller number is $c^{\text{loc}}_{\text{RS}}(q)$.
But for $q\geq 4$, the first transition is at $c^*(q)$.
This means that there is a replica symmetry breaking and/or a discontinuous phase transition.

\begin{proofof}{\bf Proof of Corollary \ref{cor:ANSK}:}
Define the quenched entropy density to be
$$
s_N(\beta,c)\, =\, \E_{N,c}\left[\mathfrak{s}_N(\beta \bJ)\right]\, .
$$
Then we know that it is still nonnegative, since it is the expectation of a pointwise nonnegative random variable.
Also, for each $\delta$,
$$
s_N(\beta,c)\, \leq\, \frac{\beta}{\delta} \left[p_N(\beta-\delta,c) - p_N(\beta,c)\right] - p_N(\beta,c)\, .
$$
This means that
$$
0\, \leq\, \frac{\beta}{\delta}\left[p(\beta-\delta,c) - p(\beta,c)\right] - p(\beta,c)\, .
$$
But if $\beta\leq\beta^*(q,c)$ then $p(\beta,c)=\mathscr{P}(\beta,c)$ and $p(\beta-\delta,c) = \mathscr{P}(\beta-\delta,c)$.
Since $\mathscr{P}(\beta,c)$ is analytic in $\beta$, we make take a derivative. In other words, we make take the limit $\delta\downarrow0$.
Therefore, we find
$$
\beta\leq\beta^*(q,c)\quad \Rightarrow\quad 0\, \leq\, \beta\, \frac{\partial}{\partial \beta} \mathscr{P}(\beta,c) - \mathscr{P}(\beta,c)\, ,
$$
which is the condition from the statement of the corollary.
\end{proofof}

\subsection{Large deviation problem}

We now return to the proof of Theorem \ref{thm:AN}
with 
%
an elementary calculation.
\begin{lemma}
\label{lem:cond}
Given an $R$-replica configuration $\Sigma_R = (\s^{(1)},\dots,\s^{(R)}) \in q^{N\times R}$, let 
$H_N(\Sigma_R,\bJ)$ denote the sum of $H_N(\s^{(r)},\bJ)$ for $r=1,\dots,R$. Then
$$
\E_{N,c}\left[ e^{-\beta H_N(\Sigma_R,\bJ)}\, \big|\, \{|\bJ| = K\}\right]\,
=\, \left(\frac{1}{N^2} \sum_{i,j=1}^{N} e^{-\beta \sum_{r=1}^{R} \delta(\s^{(r)}_i,\s^{(r)}_j)} \right)^K\, ,
$$
for each $K \in \N_0$.
\end{lemma}

\begin{proof}
This is merely the concatenation of two elementary and well-known results: for $R$ independent Poisson random variables, $X_1,\dots,X_R$,
with expectations $\lambda_1,\dots,\lambda_R$,
conditioning on the event $\{X_1+\dots+X_R=K\}$ results in the multinomial distribution:
$$
\P(\{X_1=k_1\, ,\ \dots\, ,\ X_r\, =\, k_R\}\, |\, \{X_1+\dots+X_R=k_1+\dots+k_R\})\,
=\, \frac{K!}{\prod_{r=1}^{R} k_r!} \prod_{r=1}^{R} \left(\frac{\lambda_r}{\lambda_1+\dots+\lambda_R}\right)^{k_r}\, ,
$$
and the binomial (or more generally multinomial) formula.
\end{proof}

In order to carry out the second moment calculation, let introduce the
restricted the partition function.
For $N = \mathcal{N} q$, with $\mathcal{N} \in \N$, let 
$$
[q]^{(\mathcal{N},q)}\, =\, \{\s \in [q]^{\mathcal{N}q}\, :\, \rho_{\sigma}(s) = q^{-1} \text{ for each $s \in [q]$}\}\, .
$$
These are the ``balanced'' configurations.
%
The constrained partition function will be defined
$$
\tilde{Z}_{(\mathcal{N},q)}(\bJ)\, =\, \sum_{\s \in [q]^{(\mathcal{N},q)}} e^{-\beta H_N(\s,\bJ)}\, .
$$
Let  $\mathcal{M}([q]^2)$ be the set of all probability measures on $[q]^2$, 
and let $\mathcal{M}_*([q]^2)$ denote the subset of those measures
such that the marginal on both factors in $[q]\times [q]$ are uniform.
Finally, let $\mathcal{M}_*([q]^2,N)$ denote the set of all $\mu \in \mathcal{M}_*([q]^2)$
such that $N\mu(\{(r_1,r_2)\})$ is an integer for each $(r_1,r_2) \in [q]^2$.
Then the following formulas immediately follow from Lemma \ref{lem:cond}:
$$
\E_{N,c}\left[\tilde{Z}_{(\mathcal{N},q)}(\beta \bJ)\, \big|\, \{|\bJ|=K\}\right]\,
=\, \left|[q]^{(\mathcal{N},q)}\right| e^{K \ln\left(1 - \frac{1-e^{-\beta}}{q}\right)}\,
$$
and
$$
\E_{N,c}\left[\tilde{Z}_{(\mathcal{N},q)}(\beta\bJ)^2\, \big|\, \{|\bJ|=K\}\right]\,\\
=\, \sum_{\mu \in \tilde{\Delta}_N([q]^2)} \frac{N!}{\prod_{(r_1,r_2) \in [q]^2} [N \mu(\{(r_1,r_2)\})]!}
e^{K w(\beta,q,\mu)}\, ,
$$
where $w(\beta,q,\mu) = \ln\left(1- 2(1-e^{-\beta})q^{-1} + (1-e^{-\beta})^2 \sum_{(r_1,r_2) \in [q]^2} \mu(\{(r_1,r_2)\})^2\right)$.

\begin{corollary} 
\label{cor:LD}
For any $\kappa>0$, and writing $N=\mathcal{N}q$,
\begin{gather*}
\lim_{\substack{\mathcal{N} \to \infty\\ K/N \to \kappa/2}} \frac{1}{N}\,
\ln \E_{N,c}\left[\tilde{Z}_{(\mathcal{N},q)}(\beta,\bJ)\, \big|\, \{|\bJ|=K\}\right]\,
=\, \mathscr{P}(\beta,\kappa,q)\quad  \text{ and }\\
\lim_{\substack{\mathcal{N} \to \infty\\ K/N \to \kappa/2}} \frac{1}{N}\,
\ln \E_{N,c}\left[\tilde{Z}_{(\mathcal{N},q)}(\beta,\bJ)^2\, \big|\, \{|\bJ|=K\}\right]\,
=\, \max_{\mu \in \mathcal{M}_*([q]^2)} \phi^{(2)}(\beta,\kappa,q,\mu)
\quad \text{ where}\\
\phi^{(2)}(\beta,\kappa,q,\mu)\, =\, s(\mu) + \frac{\kappa}{2}\, \ln\left(1 - \frac{2(1-e^{-\beta})}{q} + (1-e^{-\beta})^2
\sum_{(r_1,r_2) \in [q]^2} \mu(\{(r_1,r_2)\})^2\right)\, .
\end{gather*}
\end{corollary}
Here we have used the symbol $s(\mu)$ for the entropy
$$
s(\mu)\, =\, - \sum_{(r_1,r_2) \in [q]^2} \mu(\{(r_1,r_2)\}) \ln \mu(\{(r_1,r_2)\})\, .
$$

\begin{proof}
Both follow from Stirling's formula, the previous formulas and Varadhan's Lemma or rather Laplace's method which suffices.
See for example, \cite{DemboZeitouni}.
\end{proof}

With this set-up, we will see that the condition to use the second moment method at parameters $(\beta,c,q)$ is
$$
\max_{\mu \in \mathcal{M}_*([q]^2)} \phi^{(2)}(\beta,q,c,\mu)\, =\, 2 \mathscr{P}(\beta,c,q)\, .
$$
(The variable $\kappa$ is serving as a placeholder for $c$ at present.)
In light of Theorem \ref{thm:Annealed} we also see that the condition is that $\rho_{\Sigma_R}(s) = q^{-R}$ for all choices
of $s=(s_1,\dots,s_R) \in [q]^R$. This merely restates the fact that the empirical measure must collapse on the uniform measure.
If the optimizer of the large deviation principle $\mu \in \mathcal{M}([q]^2)$, one does recover the result above.

We solve this problem by first analyzing the $q=2$ case.
We note that
generally speaking, for all $q$,
\begin{equation}
\label{eq:phiDiff}
\begin{split}
\phi^{(2)}(\beta,c,q,\mu) - 2 \mathscr{P}(\beta,c,q)\, &=\,  - \sum_{(r_1,r_2) \in [q]^2} \mu(\{(r_1,r_2)\}) \ln[q^2 \mu(\{(r_1,r_2)\})]\\
&\qquad
+  \frac{c}{4}\, \ln\Bigg(1 - x^2 \sum_{(r_1,r_2) \in [q]^2} [q^2 \mu(\{(r_1,r_2)\}) - 1]^2\Bigg)\, , 
\end{split}
\end{equation}
where $x=x(\beta,q)=(1-e^{-\beta})/(q-1+e^{-\beta})$, as was defined in Section \ref{subsec:KZ}. 
For $q=2$, one may control the second term on the right hand side by a linearization.

\subsection{Ising case: $q=2$}

For $q=2$ by the conditions on the marginals,
we  may parametrize $\mu \in \mathcal{M}_*([q]^2)$ by a single number 
$$
\theta\, =\, 2\mu(\{(1,1)\})\, =\, 2\mu(\{(2,2)\})\, ,\qquad
1-\theta\, =\, 2\mu(\{(1,2)\})\, =\, 2\mu(\{(2,1)\})\, .
$$
Hence, using (\ref{eq:phiDiff}) and the linearization inequality $\ln(1-x) \leq -x$,
$$
\phi^{(2)}(\beta,c,2,\mu) - 2 \mathscr{P}(\beta,c,2)\, \leq\,  - \theta \ln(2\theta) - (1-\theta) \ln[2(1-\theta)] 
-  cx^2\, (2\theta-1)^2\, .
$$
The right hand side is directly related to the large deviation problem for the mean field Ising/$q=2$ Potts model, i.e., the
Curie-Weiss model.
From this it is easy to see that the stability of the symmetric ansatz is $cx^2\leq 1$.

In other words, when this condition is satisfied the right hand side has no critical point other than $\theta=1/2$.
So the unique maximizer of the right hand side is $\theta=1/2$, at which point the right hand side equals zero.
So the left hand side is always bounded above by $0$.
This means that
$$
\max_{\mu \in \mathcal{M}_*([q]^2)} \phi^{(2)}(\beta,c,2,\mu)\, =\, 2 \mathscr{P}(\beta,c,2)\, ,
$$
because the totally symmetric choice of $\mu$, corresponding here to $\theta=1/2$, does give equality.
As we will see in the next section, this condition guarantees $p(\beta,c,q) = \mathscr{P}(\beta,c,q)$, in this case for $q=2$.
\label{subsec:Ising}

\subsection{The 
optimization principle}


In Theorem 9 of  \cite{AchlioptasNaor}, 
it is proved a general result which implies
that, to find the maximizer of $\phi^{(2)}(\beta,c,q,\mu)$ among all $\mu \in \mathcal{M}_1([q]^2)$, it suffices to consider 
a very restricted subclass.
For $k\in\{0,1,\dots,q\}$ and $t$ a real number satifying $0\leq t\leq q$, define the measure $\mu_{k,t}$ where
$$
\mu_{k,t}(\{(r_1,r_2)\})\,
=\, \begin{cases}
q^{-2}\, , & \text{ if $r_1 \leq k$,}\\
tq^{-2}\, , & \text{ if $r_1 >k$ and $r_2=1$,}\\
\frac{q-t}{q-1}\, \cdot q^{-2}\, , & \text{ if $r_1>k$ and $r_2>1$.}
\end{cases}
$$
It follows the result that one may restrict the optimizer of $\phi^{(2)}$ among all $\mu \in \mathcal{M}_1([q]^2)$
to this subset of measures.
Denote $\Phi^{(2)}(\beta,c,q,k,t) = \phi^{(2)}(\beta,c,q,\mu_{k,t})$. Then direct calculation shows
$$
\Phi^{(2)}(\beta,c,q,k,t) -
2 \mathscr{P}(\beta,c,q)\,
=\,
\frac{c}{2}\, \ln\left(1 + \frac{x^2q(q-k)(t-1)^2}{(q-1)^3}\right)
- \frac{(q-k)\left[t \ln t + (q-t) \ln\left(\frac{q-t}{q-1}\right)\right]}{q^2}\, ,
$$
where $x=x(\beta,q)$ is as defined before in Section \ref{subsec:KZ}.
Note that in this formula $k$ now appears as a parameter. So, since we are just trying to optimize
this quantity, it suffices to take $k$ real in $[0,q]$.
Then, defining $\mathfrak{C}(\beta,q,c)$ and $\mathfrak{K}(\beta,q,k)$ through
$$
\left(q -\mathfrak{K}(\beta,q,k)\right)\, =\, x^2q^2(q-k)\quad \text { and } \quad
\mathfrak{C}(\beta,q,c)\, =\, x^2 q^2 c\, ,
$$
we see that modulo an overall multiplier, the difference is actually equal to the zero temperature quantity with
$c$ rescaled to $\mathfrak{C}$ and the ``real parameter'' $k$ rescaled to $\mathfrak{K}$
$$
x^2q\left(\Phi^{(2)}(\beta,c,q,k,t) -
2 \mathscr{P}(\beta,c,q)\right)\,
=\, \Phi^{(2)}(\infty,\mathfrak{C}(\beta,q,c),q,\mathfrak{K}(\beta,q,k),t) - 2 \mathscr{P}(\infty,q,\mathfrak{C}(\beta,q,c))\, .
$$
We are looking for the optimal choice of real parameters $k,t \in [0,q]$ in the left hand side,
and the positive multiplier $q x^2$ does not affect the arg-max.
Moreover, 
in Theorem 7 of \cite{AchlioptasNaor}, it was studied the $\beta=\infty$ problems, which is the right hand side.
In the notation of our present context, \cite{AchlioptasNaor} showed that as long as $\mathfrak{C} \leq 2 q \ln q$,
the optimal choice of $t$ is $t=1$, which is the symmetric point.
In fact, for $t=1$ the value of $\mathfrak{K}$ becomes irrelevant, because all measures with $t=1$ are the same.
Therefore, using results of  \cite{AchlioptasNaor}, we can finish the proof of our result.

\begin{proofof}{\bf Proof of Theorem \ref{thm:AN}:}
Using the condition $\mathfrak{C}(\beta,q,c) \leq 2 q \ln q$, and then using $\kappa$ as a place-holder for $c$ momentarily,
the results above in conjunction with Corollary \ref{cor:LD} imply that if $\kappa \leq 2 \ln(q) / (qx^2)$, then 
$$
\lim_{\substack{\mathcal{N} \to \infty\\ K/N \to \kappa/2}} \frac{1}{N}\,
\ln \E_{N,c}\left[\tilde{Z}_{(\mathcal{N},q)}(\beta \bJ)^2\, \big|\, \{|\bJ|=K\}\right]\,
=\, 2 
\lim_{\substack{\mathcal{N} \to \infty\\ K/N \to \kappa/2}} \frac{1}{N}\,
\ln \E_{N,c}\left[\tilde{Z}_{(\mathcal{N},q)}(\beta \bJ)\, \big|\, \{|\bJ|=K\}\right]\, ,
$$
and the right hand side equals $2 \mathscr{P}(\beta,c,q)$.
By H\"older's inequality and convexity generally we know that $\E[\ln X] \geq 2 \ln \E[X] - \frac{1}{2} \ln \E[X^2]$,
for any nonnegative random variable $X$.
Therefore,
$$
\lim_{\substack{\mathcal{N} \to \infty\\ K/N \to \kappa/2}} \frac{1}{N}\,
\E_{N,c}\left[\tilde{Z}_{(\mathcal{N},q)}(\beta \bJ)\, \big|\, \{|\bJ|=K\}\right]\,
\geq\, \mathscr{P}(\beta,\kappa,q)\, .
$$
Concentration of measure may be established for $|\bJ|$ in the measure $\P_{N,c}$.
Namely $|\bJ|/N$, the variable we have called $\kappa/2$ up to now, concentrates around $c/2$.
Moreover, one can show that the conditional expectation is Lipschitz as a function of $\kappa$ by general
principles.
Therefore, this establishes the lower bound in the limit as $N \to \infty$ along integer multiples of $N$,
$p(\beta,c,q) \geq \mathscr{P}(\beta,c,q)$.
In other words, since we know the limit of $p_N(\beta,c,q)$ exists as $N \to \infty$, it does not
matter what subsequence one takes to obtain that limit.
This lower bound matches the upper bound so it gives the identity.

For $q=2$ the argument is the same except that 
we use the analysis of Subsection \ref{subsec:Ising}.
\end{proofof}

\appendix
\section{Appendix}

\subsection{Interpolation Results}
\label{sec:ProofsInterpolation}

The proofs in this subsection are all based on 
Lemma \ref{lem:interpo}.
As a first step, we note that by using the series expansion in the radius of convergence of $\ln(1-x)$, 
\begin{equation}
\label{eq:interpo2}
\frac{d}{dt}\,  p_N(\beta,\bc(t))\, 
=\, -\sum_{R=0}^{\infty} \frac{(1-e^{-\beta})^R}{2N^2R}\, \sum_{s \in [q]^R} \sum_{i,j=1}^{N} \frac{dc_{ij}}{dt}\, 
\left\langle\hspace{-5pt}\left\langle 
\prod_{r=1}^{R} \delta(\s_i^{(r)},s_r) \delta(\s_j^{(r)},s_r)
\right\rangle\hspace{-5pt}\right\rangle_{N,\beta,\bc(t)}\, .
\end{equation}

\begin{proofof}{\bf Proof of Corollary \ref{cor:superadd}:}
Let $N=N_1+N_2$.
Let us define $\bc^{(0)}$ and $\bc^{(1)}$ such that $p_N(\beta,\bc^{(0)})$ and $p_N(\beta,\bc^{(1)})$ are the left- and right-hand-sides
of (\ref{eq:pNsuperadd}), respectively:
$$
c_{ij}^{(0)}\, \equiv\, c\, ,\ \text{ for all $i,j$, }\ \text{ and }\quad
c_{ij}^{(1)}\, =\, \frac{N}{N_1}\, \boldsymbol{1}_A(i) \boldsymbol{1}_A(j) + \frac{N}{N_2}\,  \boldsymbol{1}_B(i) \boldsymbol{1}_B(j)\, ,
$$
where denote two sets, $A=\{1,\dots,N_1\}$, $B = \{N_1+1,\dots,N\}$.
We remind the reader of the notation introduced in (\ref{eq:qRreplica}).
Let us extend this as follows
$$
\rho_{\Sigma_R}^A(s)\, =\, \frac{1}{N_1}\, \sum_{i\in A} \prod_{r=1}^{R}\delta(\s^{(r)}_i,s_r)\quad \text{ and }\quad
\rho_{\Sigma_R}^B(s)\, =\, \frac{1}{N_2}\, \sum_{i\in B} \prod_{r=1}^{R}\delta(\s^{(r)}_i,s_r)\, .
$$
Then using  (\ref{eq:interpo2}) for $\bc(t) = (1-t) \bc^{(0)} + t \bc^{(1)}$, 
$$
\frac{d}{dt}\, p_N(\beta,\bc(t))\, 
=\,  - \frac{N_1 N_2}{2N^2}\, \sum_{R = 0}^{\infty} \frac{(1-e^{-\beta})^R}{R}\, \sum_{s \in [q]^R} \left\langle\hspace{-3pt}\left\langle \left[\rho^A_{\Sigma_R}(s)
- \rho^B_{\Sigma_R}(s)\right]^2 \right\rangle\hspace{-3pt}\right\rangle_{N,\beta,\bc(t)}\, .
$$
This has a definite sign, which implies $p_N(\beta,\bc^{(0)}) \geq p_N(\beta,\bc^{(1)})$.
\end{proofof}

\begin{proofof}{\bf Proof of Lemma \ref{lem:superadd}:}
To recapitulate Fekete's argument,  for any $M\leq N$,
$$
x_N\, \geq\, \frac{M\lfloor N/M \rfloor}{N}\, x_M + \frac{N - M \lfloor N/M \rfloor}{N} x_{N - M \lfloor N/M \rfloor}\, 
\geq\, \frac{M\lfloor N/M \rfloor}{N}\, x_M + \frac{1}{N}\, \min_{k\leq M-1} k x_k\, ,
$$ 
which shows that $\liminf_{N \to \infty} x_N\, \geq\,  x_M$ for each $M$. Hence 
$\liminf_{N\to\infty} x_N\, \geq\, \sup_{M \to \infty} x_M$, and the supremum is no less
than the limit superior, $\limsup_{M \to \infty}  x_M$.
Of course the opposite inequality $\limsup_{M \to \infty} x_M \geq \liminf_{N \to \infty} x_N$, holds by definition.
So Fekete's argument is complete. 

Let $x_* = \sup_{N} x_N$.
Defining $\mathcal{N}_{N,K} = NK$,
$$
\frac{\mathcal{N}_{N,K}-M}{\mathcal{N}_{N,K}}\, x_{\mathcal{N}_{N,K}} - \frac{M}{\mathcal{N}_{N,K}}\, x_M\,
\geq\, \frac{1}{K}\, \sum_{k=1}^{K} \left(\frac{\mathcal{M}_{M,N,k} + N}{N}\, x_{\mathcal{M}_{M,N,k} +N} - \frac{\mathcal{M}_{M,N,k}}{N}\, x_{\mathcal{M}_{M,N,k}}\right)\, ,
$$
where we define $\mathcal{M}_{M,N,k} = M + N (k-1)$. The left hand side converges to $x_*$ as $K \to \infty$,
while the right hand side is uniformly bounded below by
$$
\inf_{M' \geq M} \left(\frac{M'+N}{N}\, x_{M'+N} - \frac{M'}{N}\, x_{M'}\right)\, .
$$
So taking the limit $M \to \infty$ of the last expression, and then taking the limit superior as $N \to \infty$,
$$
x_*\, \geq\, \limsup_{N \to \infty} \liminf_{M \to \infty} \left(\frac{M+N}{N}\, x_{M+N} - \frac{M}{N}\, x_M\right)\, .
$$
But for any $M$, $(\frac{M+N}{N}\, x_{M+N} - \frac{M}{N}\, x_M) \geq x_N$, by superadditivity again,
and taking the limit superior of this gives $x_*$. So $x_*$ is not only an upper bound for the right hand side
of the equation displayed above, it is a lower bound, too.
\end{proofof}

\begin{proofof}{\bf Proof of Theorem \ref{thm:Annealed}:}
Using (\ref{eq:interpo2}) and the definition of $\rho_{\Sigma_R}(s)$ from (\ref{eq:qRreplica}), we can rewrite
$$
\frac{d}{dt}\, p_N(\beta,ct)\,
=\, -\frac{c}{2}\, \sum_{R=0}^{\infty} \frac{(1-e^{-\beta})^R}{2N^2R}\, 
\sum_{s \in [q]^R} \left\langle\hspace{-3pt}\left\langle \rho_{\Sigma_R}(s)^2 \right\rangle\hspace{-3pt}\right\rangle_{N,\beta,ct}\, .
$$
Therefore, using the fact that $\sum_{s \in [q]^R} \rho_{\Sigma_R}(s) = 1$ for every $\Sigma_R \in [q]^{N \times R}$, we see that
$$
\frac{d}{dc}\left(\mathscr{P}(\beta,c) - p_N(\beta,c)\right)\, 
=\, \frac{1}{2} \sum_{R=0}^{\infty} \frac{(1-e^{-\beta})^R}{R}\, \sum_{s \in [q]^R} 
\sum_{s \in [q]^R} \left\langle\hspace{-3pt}\left\langle \left[\rho_{\Sigma_R}(s) - q^{-2}\right]^2 \right\rangle\hspace{-3pt}\right\rangle_{N,\beta,c}\, .
$$
Using the fact that $p_N(\beta,0) = \mathscr{P}(\beta,0) = \ln(q)$, and integrating, this proves the result.
\end{proofof}

\begin{proofof}{\bf Proof of Corollary \ref{cor:LipGeneralp}:}
Apply (\ref{eq:interpo2}) to $\bc(t) = (1-t) \bc^{(1)} + t \bc^{(2)}$, defined for $t \in [0,1]$, to obtain the bound
$\left|\frac{d}{dt}\, p_N(\beta,\bc(t))\right|\, \leq\, \frac{\beta}{2N^2}\, \sum_{i,j=1}^{N} \left|c_{ij}^{(2)} - c_{ij}^{(1)}\right|$,
and then integrate. The result for numbers follows by taking $\bc^{(1)}$ and $\bc^{(2)}$ such that $c_{ij}^{(1)}=c_1$ and $c_{ij}^{(2)} = c_2$ for all $i,j$.
\end{proofof}

\subsection{Extended variational principle bounds}
\label{sec:Cavity}

The proofs in this section are more involved than the previous section. In order to prove bounds
and the extended variational principle,
as stated, we first generalize the definition of random spin structure and cavity field functionals.
The definition we give is based on the sampling-resampling definition Kingman eventually gave for his random partition structures \cite{Kingman1,Kingman2}.
Recall from Section \ref{sec:EVP} that $[q]^{\N}$ denotes the set of all infinite spin configurations $\tau = (\tau_1,\tau_2,\dots)$ with each $\tau_n \in [q]$.
Let $[q]^{\N \times \N}$ the set of all infinite sequences of replicas $\mathcal{T} = (\tau^{(1)},\tau^{(2)},\dots)$ with
each $\tau^{(r)} \in [q]^{\N}$.
With the product topology this is compact and metrizable.
Let $\mathscr{M}^*$ denote the set of all Borel probability measures on $[q]^{\N \times \N}$.
We will denote such measures as $\mathfrak{L} \in \mathscr{M}^*$ in order to distinguish them from
measures $\mathcal{L} \in \mathscr{M}$.
Let $\mathscr{S}^*$ denote the set of all measures $\mathfrak{L} \in \mathscr{M}^*$ satisfying the following two exchangeability
conditions:
\begin{itemize}
\item[(i)] For any measurable set $A \subseteq [q]^{\N \times \N}$ and any $\pi \in S_{\infty}$,
$$
\mathfrak{L}(\{\mathcal{T}\, :\, (\tau^{(1)}\circ \pi, \tau^{(2)}\circ \pi,\dots) \in A\})\, =\, \mathfrak{L}(A)\, .
$$
\item[(ii)] With the same setup, $\mathfrak{L}(\{\mathcal{T}\, :\, (\tau^{(\pi(1))},\tau^{(\pi(2))},\dots) \in A\}) = \mathfrak{L}(A)$.
\end{itemize}
The $R$-replica cavity field functions are
\begin{gather*}
G^{(1)}_{N,R}(\beta,c,\mathfrak{L})\,
=\, \int_{[q]^{\N\times \N}}
\widetilde{\E}_{N,c}\left[\frac{1}{N} \ln \sum_{r=1}^{R} \frac{1}{R} \sum_{\s \in [q]^N} \exp\left(-\beta \widetilde{H}_N(\bI,\tau^{(r)},\sigma)\right)\right]\,  d\mathfrak{L}(\mathcal{T})\quad \text{ and}\\[5pt]
G^{(2)}_{N,R}(\beta,c,\mathfrak{L})\,
=\, \int_{[q]^{\N\times \N}} \sum_{K=0}^{\infty} \frac{\pi_{cN/2}(K)}{N}\, \ln\left( \sum_{r=1}^{R} \frac{1}{R} \exp\left(-\beta \sum_{k=1}^{K} \delta(\tau^{(r)}_{2k-1},\tau^{(r)}_{2k})\right)\right)\, 
d\mathfrak{L}(\mathcal{T})\, .
\end{gather*}
\begin{lemma}
\label{lem:continuity}
For any $\beta ,c \geq 0$, $N\in \N$ and $\mathfrak{L} \in \mathscr{S}^*$, the following limits exist and determine continuous
functions on $\mathscr{S}^*$:
$$
\widetilde{G}^{(1)}_{N}(\beta,c,\mathfrak{L})\, \stackrel{{\rm def}}{:=}\, \lim_{R \to \infty} G^{(1)}_{N,R}(\beta,c,\mathfrak{L})\quad \text{ and } \quad 
\widetilde{G}^{(2)}_{N}(\beta,c,\mathfrak{L})\, \stackrel{{\rm def}}{:=}\, \lim_{R \to \infty} G^{(2)}_{N,R}(\beta,c,\mathfrak{L})\, .
$$
\end{lemma}
We do not use the continuity in the sequel.
In order to set up a general proof for both, let us define
$$
g^{(1)}_{\bI}(\tau)\, =\, \sum_{\s \in [q]^N} \exp\left(-\beta \widetilde{H}_N(\bI,\tau,\s)\right)\quad
\text{ and} \quad
g^{(2)}_{K}(\tau)\, =\, \exp\left(-\beta \sum_{k=1}^{K} \delta(\tau_{2k-1},\tau_{2k})\right)\, .
$$
Note that $q^N e^{-\beta|\bI|} \leq g^{(1)}_{\bI} \leq q^N$ and $e^{-\beta K} \leq g^{(2)}_K \leq 1$.

\begin{proof}
For simplicity, suppose that $g : [q]^{\N} \to (0,\infty)$ is a continuous, positive function, with bounds $0<m\leq g\leq M<\infty$.
Given a finite set $A \subset \N$, let $g_A(\mathcal{T}) = \sum_{r \in A} |A|^{-1} g(\tau^{(r)})$.
Define $\gamma_A(\mathcal{T}) = \ln g_A(\mathcal{T})$.
This is continuous on $[q]^{\N \times \N}$ for each $R \in \N$.
Define
$\Gamma_R(\mathfrak{L}) = \int_{[q]^{\N\times \N}} \gamma_{\{1,\dots,R\}}(\mathcal{T})\, d\mathfrak{L}(\mathcal{T})$,
which is continuous on $\mathcal{M}^*$. We want to show that, when restricted to $\mathscr{S}^*$ these functions
converge pointwise to a continuous limit. From this we will be able to prove the desired result.

Let $\mathcal{A}(R) = \{1,\dots,R\}$ and let $\mathcal{A}(R,r) = \mathcal{A}(R) \setminus \{r\}$, for $r=1,\dots,R$.
Then, leaving out the explicit dependence on the argument 
$$
\gamma_{\mathcal{A}(R)} - \frac{1}{R}\, \sum_{r \in \mathcal{A}(R)} \gamma_{\mathcal{A}(R,r)}\,
=\, \frac{1}{R}\, \int_0^1 \sum_{r=1}^R \frac{g_{\mathcal{A}(R)} - g_{\mathcal{A}(R,r)}}
{(1-\theta) g_{\mathcal{A}(R)} - \theta g_{\mathcal{A}(R,r)}}\, d\theta\, .
$$
Note that the left hand side is nonnegative by Jensen's inequality.
Direct inspection shows the integrand in the right hand side is zero when $\theta=0$.
Therefore, using integrating by parts, the right hand side equals
$$
\frac{1}{R}\, \sum_{r=1}^{R} \int_0^1 \left( \int_0^{\theta_1} \left(\frac{g_{\mathcal{A}(R)} - g_{\mathcal{A}(R,r)}}
{(1-\theta) g_{\mathcal{A}(R)} - \theta g_{\mathcal{A}(R,r)}}\right)^2\, d\theta\,\right)\, d\theta_1\, 
\leq\, \frac{1}{2m^2R}\,  \sum_{r=1}^{R} \left(g_{\mathcal{A}(R)} - g_{\mathcal{A}(R,r)}\right)^2\, .
$$
But a simple calculation shows that
$$
\frac{1}{R}\,  \sum_{r=1}^{R} \left(g_{\mathcal{A}(R)} - g_{\mathcal{A}(R,r)}\right)^2\, 
=\, \frac{1}{(R-1)^2} \left[\frac{1}{R}\, \sum_{r=1}^{R} \left(g(\tau^{(r)}) - g_{\mathscr{A}(R)}(\mathcal{T})\right)^2\right]\,
\leq\, \frac{(M-m)^2}{(R-1)^2}\, .
$$
The key point is that this is summable, summing over $R \in \{2,3,\dots\}$. Also, by exchangeability, taking expectations shows that this gives
$$
0\, \leq\, \Gamma_R(\mathfrak{L}) - \Gamma_{R-1}(\mathfrak{L})\, \leq\, \frac{(\frac{M}{m}-1)^2}{2(R-1)^2}\, .
$$
Since the uniform limit of continuous functions is continuous this shows that $\Gamma = \lim_{R \to \infty} \Gamma_R$ is continuous.

Now for the lemma as stated there is the technicality that the functions depend on parameters, $|\bI|$ and $K$, and that the ratio $M/m$ diverges
as these parameters do. On the other hand, the quantity on the right hand side above is integrable against the measures for these random
parameters. More specifically, one has integrability of $(e^{\beta|\bI|}-1)^2$ and $(e^{\beta K}-1)^2$ against the appropriate Poisson measures.
So the dominated convergence theorem shows that the result still holds.
\end{proof}

Given $\mathcal{L} \in \mathscr{S}$, one may define $\mathfrak{L}_{\mathcal{L}} \in \mathscr{S}^*$ as follows.
Suppose that $(\bxi,\mathcal{T})$ is distributed according to $\mathcal{L}$.
Let $\mathfrak{r}_1,\mathfrak{r}_2,\dots$ be i.i.d., $\N$-valued random variable distributed according to $\bxi$.
Then we let $\mathfrak{L}_{\mathcal{L}}$ be the marginal distribution of $(\tau^{(\mathfrak{r}_1)},\tau^{(\mathfrak{r}_2)},\dots)$.

\begin{corollary}
\label{cor:Identity}
For any $\mathcal{L} \in \mathscr{S}$, $\widetilde{G}_N^{(i)}(\beta,c,\mathfrak{L}_{\mathcal{L}}) = G_N^{(i)}(\beta,c,\mathcal{L})$
for $i=1,2$.
\end{corollary}
\begin{proof}
If one fixes $\bI$ and $K$, then this follows from the weak law of large numbers for $\mathfrak{r}_1,\mathfrak{r}_2,\dots$, and continuity of the functions involved.
This establishes convergence, pointwise for each finite $|\bI|$ and $K$.
For random $\bI$ and $K$ one can use the dominated convergence theorem, since the functions satisfy
exponential bounds with respect to $|\bI|$ and $K$,
and these random variables are Poissonian.
\end{proof}
\begin{proofof}{\bf Proof of Theorem \ref{thm:RSBbd}:}
In view of Corollary \ref{cor:Identity}, it suffices to prove 
$$
p_N(\beta,c) + {G}_{N,R}^{(2)}(\beta,c,\mathfrak{L})\, \leq\, {G}_{N,R}^{(1)}(\beta,c,\mathfrak{L})\, ,
$$
for each $\mathfrak{L} \in \mathscr{S}^*$ and each $R \in \N$.

In order to prove this we make yet another definition. For $M \in \N$ finite, let us define $\widetilde{\bI}$ to be a sequence
$((I_k^1,I_k^2))_{k=1}^{|\bI|}$ where $\bI$ is still $\pi_{cN}$ distributed, and conditional on that $(I_1^1,\dots,I_{|\bI|}^1)$
and $(I_1^2,\dots,I_{|\bI|}^2)$ are indepedent, with the first being uniform on $\{1,\dots,N\}$ and the second being uniform on
$\{1,\dots,M\}$. Then we define $H^{(1)}_N(\widetilde{\bI},\s,\t) = \sum_{k=1}^{|\widetilde{\bI}|} \delta(\s_{I_k^1},\tau_{I_k^2})$.
We let $\widetilde{G}_{N,R,M}^{(1)}(\beta,c,\mathfrak{L})$ be the result of changing $G_{N,R,M}^{(1)}(\beta,c,\mathfrak{L})$
by replacing $H_N(\bI,\s,\t^{(r)})$ by $\widetilde{H}_N^{(1)}(\bI,\s,\t^{(r)})$.
We notice that if $\widetilde{\bI}$ happens to be such that no two numbers in $(I_1^2,\dots,I_{|\bI|}^2)$ are the same,
then the two definitions are equal in distribution because of exchangeability of $\mathfrak{L}$.
But conditional on $|\bI|$, this happens with probability $(1-M^{-1})^{|\widetilde{\bI}|}$ which implies
$$
\lim_{M \to \infty} \widetilde{G}_{N,R,M}^{(1)}(\beta,c,\mathfrak{L})\, =\, G_{N,R}^{(1)}(\beta,c,\mathfrak{L})\, ,
$$
along with
integrability and the dominated  convergence theorem.

Similarly, let $\widetilde{\bK} = ((K_i^1,K_i^2))_{i=1}^{|\widetilde{\bK}|}$ be a sequence where $|\widetilde{K}|$ is $\pi_{cN/2}$ distributed
and conditional on that $(K_1^1,\dots,K_{|\widetilde{\bK}|}^1)$ and  $(K_1^2,\dots,K_{|\widetilde{\bK}|}^2)$ 
are all independent and uniform on $\{1,\dots,M\}$.
We let $\widetilde{H}_N^{(2)}(\widetilde{\bK},\t) = \sum_{i=1}^{|\widetilde{\bK}|} \delta(\t_{K_i^1},\t_{K_i^2})$,
and use this to replace $\sum_{i=1}^{K} \delta(\t_{2i-1},\t_{2i})$ in the definition of $G^{(2)}_{N,R}(\beta,c,\mathfrak{L})$.
We call the new version $\widetilde{G}^{(2)}_{N,R,M}(\beta,c,\mathfrak{L})$.
Then if none of the $K_i^1$'s and $K_i^2$'s are repeated, there is no real difference from before, and this happens with
probability $(1-M^{-1})^{2|\widetilde{\bK}|}$. This shows
$$
\lim_{M \to \infty} \widetilde{G}_{N,R,M}^{(2)}(\beta,c,\mathfrak{L})\, =\, G_{N,R}^{(2)}(\beta,c,\mathfrak{L})\, ,
$$

Finally, we claim that 
\begin{equation}
\label{eq:CSdesid}
p_N(\beta,c) + \widetilde{G}_{N,R,M}^{(2)}(\beta,c,\mathfrak{L})\, \leq\, \widetilde{G}_{N,R,M}^{(1)}(\beta,c,\mathfrak{L})\, ,
\end{equation}
by arguments from Section \ref{sec:interpolation}.
Indeed, considering $(\s_1,\dots,\s_N,\t_1,\dots,\t_M)$ as a spin configuration, we see that the left and right are given respectively by
$\frac{M+N}{N} p_{M+N}(\beta,\bc^{(i)}) - \frac{M}{N} \ln[q]$, for $i=1,2$, where
$$
\bc^{(1)} = \frac{M+N}{N}\, \begin{bmatrix} c \mathbbm{1}_{N,N} & 0 \\ 0 & c \mathbbm{1}_{M,M} \end{bmatrix}\quad
\text{ and }\quad
\bc^{(2)} = \frac{M+N}{N}\, \begin{bmatrix} 0 & (c/2) \mathbbm{1}_{N,M} \\ (c/2) \mathbbm{1}_{M,N} & 0 \end{bmatrix}\, ,
$$
where $\mathbbm{1}_{m,n}$ is the $m\times n$ matrix with all entries equal to 1.
In particular the difference is a positive multiple of the outer product of the vector $(1,\dots,1,-1,\dots,-1)$ with the first $N$ entries equal
to $1$ and the last $M$ equal to $-1$. So following the proof of Corollary \ref{cor:superadd} one can deduce (\ref{eq:CSdesid}).
\end{proofof}

\begin{proofof}{\bf Proof of Lemma \ref{lem:EstimateEVP}:}
We will prove the formula for $G_N^{(2)}(\beta,c,\mathcal{L}_{M,N})$, first.
Let $\widetilde{I}_{M,1},\widetilde{I}_{M,2},\dots$ be chosen independently, and uniformly from $\{1,\dots,M\}$.
Using $K$ and $\widetilde{I}_{M,1},\widetilde{I}_{M,2},\dots$, we define
the random coupling matrix $\widehat{\bJ}^{(M,N)} \in \M_M(\N_0)$ such that
$$
\widehat{J}^{(M,N)}_{i,j}\, =\, \#\{k\leq K\, :\, \widetilde{I}_{M,2k-1}=i\, ,\ \widetilde{I}_{M,2k}=j\}\, .
$$
Then these variables are distributed as independent Poisson random variables all with means $cN/(2M^2)$.
Moreover, since $K$ and $\widetilde{I}_{M,1},\widetilde{I}_{M,2},\dots$ are independent of $\widetilde{\bJ}^{(M,N)}$,
this means that $\widetilde{\bJ}^{(M,N)}$ and $\widehat{\bJ}^{(M,N)}$ are independent.
Since $\widetilde{\bJ}^{(M,N)}$ has the distribution $\P_{N,\tilde{c}_{M,N}}$, the sum $\widetilde{\bJ}^{(M,N)} + \widehat{\bJ}^{(M,N)}$
has the distribution $\P_{N,\tilde{c}_{M,N} + (cN/M)}$.
This justifies the desired equation. More precisely,
$$
G^{(2)}_{N,c}\left(\beta,\mathcal{L}_{M,N}\right)\,
=\, \frac{1}{N}\, \E\left[\ln \sum_{\alpha=1}^{q^M} \omega_{\beta,M,\widetilde{\bJ}^{(M,N)}}\left({\sigma}(\alpha)\right) 
\exp\left(-\beta\sum_{k=1}^{K} \delta\left({\sigma}_{\widetilde{I}_{M,2k-1}}(\alpha),\widetilde{\sigma}_{{I}_{M,2k}}(\alpha)\right)\right)\right]\, ,
$$
where the expectation is over $\widetilde{\bJ}^{(M,N)}$, $K$ and $\widetilde{I}_{M,1},\widetilde{I}_{M,2},\dots$.
But using $\widehat{\bJ}^{(M,N)}$ this may be rewritten as
$$
G^{(2)}_{N,c}\left(\beta,\mathcal{L}_{M,N}\right)\,
=\, \frac{1}{N}\, \E\left[\ln \sum_{\alpha=1}^{q^M} \omega_{\beta,M,\widetilde{\bJ}^{(M,N)}}\left(\widetilde{\sigma}^{(\alpha)}\right) 
\exp\left(-\beta\sum_{i,j=1}^{M} \widehat{J}^{(M,N)}_{ij} \delta\left(\widetilde{\sigma}_{i}^{(\alpha)},\widetilde{\sigma}_{j}^{(\alpha)}\right)\right)\right]\, .
$$
Keeping track of the definition of ${\sigma}^{(\alpha)}$ and also of $H_{M}(\cdot,\widehat{\bJ}^{(M,N)})$, we have
$$
G^{(2)}_{N,c}\left(\beta,\mathcal{L}_{M,N}\right)\,
=\, \frac{1}{N}\, \E\left[\ln \sum_{{\sigma} \in[q]^M} \omega_{\beta,M,\widetilde{\bJ}^{(M,N)}}\left({\sigma}\right) 
\exp\left(-\beta H_{M}\left(\widetilde{\sigma},\widehat{\bJ}^{(M,N)}\right)\right)\right]\, .
$$
Finally, using the definition of the Boltzmann-Gibbs measure, this can be rewritten as 
\begin{align*}
G^{(2)}_{N,c}\left(\beta,\mathcal{L}_{M,N}\right)\,
&=\, \frac{1}{N}\, \E\left[\ln \sum_{{\sigma} \in[q]^M} \frac{1}{Z_N\left(\beta,\widetilde{\bJ}^{(M,N)}\right)}\, 
\exp\left(-\beta H_{M}\left({\sigma},\widetilde{\bJ}^{(M,N)}+\widehat{\bJ}^{(M,N)}\right)\right)\right]\\
&=\, \frac{1}{N}\, \E\left[\ln Z_N\left(\beta,\widetilde{\bJ}^{(M,N)}+\widehat{\bJ}^{(M,N)}\right) 
- \ln Z_N\left(\beta,\widetilde{\bJ}^{(M,N)}\right)\right]\, .
\end{align*}
Keeping track of the marginal distributions of $\widetilde{\bJ}^{(M,N)}$
and  $\widetilde{\bJ}^{(M,N)}+\widehat{\bJ}^{(M,N)}$ in the measure for $\E$ does give equation for  $G_N^{(2)}(\beta,c,\mathcal{L}_{M,N})$.

The derivation of the equation for  $G_N^{(1)}(\beta,c,\mathcal{L}_{M,N})$ is similar. We leave this as an exercise for the reader.
\end{proofof}

\subsection{Poisson-Dirichlet structures}
\label{sec:PD}

\begin{proofof}{\bf Proof of Lemma \ref{lem:Laplace}:}
Using equation (\ref{eq:MGFal}) we immediately have
\begin{align*}
\E\left[\exp\left(-\lambda \int_0^{\infty} x^p\, d\Xi(x)\right)\right]\,
&=\, \exp\left(-\int_0^{\infty} \left(1  - e^{-\lambda x^p}\right)\, \frac{d}{dx}(-x^{-m})\, dx\right)\\
&=\, \exp\left(-\int_0^{\infty} x^{-m}\, \frac{d}{dx}\left(- e^{-\lambda x^p}\right)\, dx\right)\, ,
\end{align*}
using integration by parts and the fact that $(1-e^{-\lambda x^p})$ converges to zero as $x \to 0$
faster than $x^{-m}$ diverges, as long as $p>m$ for the boundary term at $0$. (The boundary term at $\infty$
follows just because $x^{-m}$ converges to zero there.)
Rewriting $y = \lambda x^p$, so that $x^{-m} = \lambda^{m/p} y^{-m/p}$ gives the result.
\end{proofof}

\begin{proofof}{\bf Proof of Theorem \ref{thm:stability}:}
This follows from the conjunction of two basic facts about general Poisson processes, both of which follow easily from the moment generating functional identity definition of Poisson processes:
\begin{itemize}
\item If $\{\xi_1,\xi_2,\dots\}$ is a Poisson process with intensity $\Lambda(d\xi)$ on $(0,\infty)$,
and $(\lambda_1,\lambda_2,\dots)$ are i.i.d., $\rho$-distributed points in $(0,\infty)$ independent of $\{\xi_1,\xi_2,\dots\}$ then the 
pairs $\{(\xi_1,\lambda_1),(\xi_2,\lambda_2),\dots\}$ are a Poisson process on the quadrant $(0,\infty)^2$ with intensity measure 
$\tilde{\Lambda}(d\xi \times d\lambda) = \Lambda(d\xi) d\rho(\lambda)$.
\item If $\{(\xi_1,\lambda_1),(\xi_2,\lambda_2),\dots\}$ is a Poisson process on $(0,\infty)^2$ with intensity measure $\tilde{\Lambda}(d\xi \times d\lambda)$,
and if $\Phi : (0,\infty)^2 \to (0,\infty)^2$ is a diffeomorphism, then $\{\Phi(\xi_1,\lambda_1),\Phi(\xi_2,\lambda_2),\dots\}$ is a Poisson process
with intensity measure $\tilde{\Lambda}_\Phi$ defined as $\tilde{\Lambda}_{\Phi}(A) = \int \boldsymbol{1}_A(\Phi(\xi,\lambda))\, \tilde{\Lambda}(d\xi\times d\lambda)$.
\end{itemize}
Taking the function $\Phi(\xi,\lambda) = (\lambda \xi,\lambda)$ we see that, for the case $\Lambda = \Lambda_m$,
$$
\tilde{\Lambda}_{\Phi}(A)\, =\, \int_0^{\infty} \left(\int_0^{\infty} \boldsymbol{1}_A(\lambda \xi,\lambda) \frac{d}{d\xi}\left(-\xi^{-m}\right)\, d\xi\right)\, d\rho(\lambda)\,
=\, \int_0^{\infty} \lambda^m \left(\int_0^{\infty} \boldsymbol{1}_A(x,\lambda)\, \frac{d}{dx}\left(-x^{-m}\right)\, dx\right)\, d\rho(\lambda)\, ,
$$
using equation (\ref{eq:MGFal}). So $\tilde{\Lambda}_{\Phi}(d\xi\times d\lambda) = \lambda^{m} \Lambda_m(d\xi) d\rho(\lambda)$.
Note that the measure $c^{-m} \lambda^{m} d\rho(\lambda)$ is a probability measure. So taking the marginal of $\tilde{\Lambda}_{\Phi}$ just on the first coordinate gives
$c^{m} \Lambda_m(d\xi) = \Lambda_m(d(c^{-1}\xi))$, which equals the change of measure of $\Lambda_m$ due to the mapping $\xi \mapsto c\xi$.
\end{proofof}

\begin{proofof}{\bf Proof of Corollary \ref{cor:Z}:}
This is proved by induction, conditioning on the $\sigma$-algebra of $\widehat{\xi}_{\bm^{(L)}_{\restriction L-1}}(\ba^{(L-1)})$.
More precisely, first construct the un-normalized Poisson process associated to this normalized random partition structure,
and then use Theorem \ref{thm:stability}.
Note that to work it is essential that $m_1<m_2<\dots<m_L$ since, in the proof of Lemma \ref{lem:Laplace} above,
one does need this condition in order to have finite fractional moments.
\end{proofof}

\subsection{Replica Symmetry Breaking results}
\label{app:Calculations}

In this section we use an abbreviated notation in order to reduce the number of symbols needed.
We hope that the reader may follow the calculation, inferring the translation needed from the context.

To calculate $G_N^{(2)}$, the easier of the two parts of the cavity field functional,
we condition on $\{\xi^{1}_{\alpha_1}\}$ and on $\{\tau^1_{\alpha_1}\}$.
Since the $\tau^2_{\alpha}$'s and $X_{\alpha}$'s are i.i.d., we do not condition on them.
Then we note that
$$
{\bf P}(\tau_{\alpha,2k-1} = \tau_{\alpha,2k}\, |\, \{\tau^1_{\alpha_1,k}\})\,
=\, {\bf P}(X_1=X_2=1) {\bf 1}\{\tau^1_{\alpha_1,2k-1} = \tau^1_{\alpha_1,2k}\}
+ [1 - {\bf P}(X_1=X_2=1)] \cdot \frac{1}{q}\, .
$$
This implies that
$$
{\E}\left[e^{-m_2 \beta \delta(\tau_{\alpha,2k-1},\tau_{\alpha,2k})}\, \big|\, \{\tau^1_{\alpha_1}\}\right]\,
=\, p^2 e^{-m_2\beta \delta(\tau^1_{\alpha_1,2k-1},\tau^1_{\alpha_1,2k})} + (1-p^2) \left(1 - \frac{1-e^{-m_2\beta}}{q}\right)\, .
$$
For
$$
\lambda_{\alpha}\, =\, e^{-\beta \delta(\tau_{\alpha,2k-1},\tau_{\alpha,2k})}\, ,
$$
we have that
$$
\{\lambda_{\alpha} \xi_{\alpha}\}\, \stackrel{d}{=}\,
\{\lambda_0 \xi_{\alpha}\}\, ,
$$
where
$$
\lambda_0\, =\, {\E}\left[ {\E}\left[e^{-m_2 \beta \delta(\tau_{\alpha,2k-1},\tau_{\alpha,2k})}\, \big|\,  \{\tau^1_{\alpha_1}\}\right]^{m_1/m_2}\right]^{1/m_1}\, .
$$
So we get in general
$$
\ln \lambda_0\,
=\, \frac{1}{m_1}\, \ln \left[
\left(1-\frac{1}{q}\right)  \left(1 - \frac{(1-p^2)(1-e^{-m_2\beta})}{q}\right)^{m_1/m_2}
+ \frac{1}{q} \left(1 - \left[p^2 + \frac{1-p^2}{q}\right]\left(1-e^{-m_2\beta}\right)\right)^{m_1/m_2}\right]
$$
Because the spin fields are independent for different values of $k$, the effect of the $L$ is just to multiply this final answer.
Therefore, taking the expectation of that, and dividing by $N$ gives
$$
G_N^{(2)}\,
=\, \frac{c}{m_1}\, \ln \left[
\left(1-\frac{1}{q}\right)  \left(1 - \frac{(1-p^2)(1-e^{-m_2\beta})}{q}\right)^{m_1/m_2}
+ \frac{1}{q} \left(1 - \left[p^2 + \frac{1-p^2}{q}\right]\left(1-e^{-m_2\beta}\right)\right)^{m_1/m_2}\right]
$$

To get the replica symmetric ansatz, we use the 2-level RPC and take the limits $m_2 \uparrow 1$ and $m_1 \downarrow 0$.
Taking $m_2 \uparrow 1$,  gives
$$
\ln \lambda_0\,
=\, \frac{1}{m_1}\, \ln
 \left[
\left(1-\frac{1}{q}\right) \left(1 - \frac{(1-p^2)(1-e^{-\beta})}{q}\right)^{m_1}
+ \frac{1}{q} \left(1 - \left[p^2 + \frac{1-p^2}{q}\right]\left(1-e^{-\beta}\right)\right)^{m_1}\right]\, .
$$
Then taking $m_1 \downarrow 0$ gives
$$
\ln \lambda_0\,
=\,  \left(1-\frac{1}{q}\right)\, \ln \left(1 - \frac{(1-p^2)(1-e^{-\beta})}{q}\right)
+ \frac{1}{q}\, \ln \left(1 - \left[p^2 + \frac{1-p^2}{q}\right]\left(1-e^{-\beta}\right)\right)\, .
$$
Thus the replica symmetric value of the first term is
\be
\label{eq:G2RS}
G_N^{(2)}\,
=\,  \frac{c}{2}\, \left(1-\frac{1}{q}\right) \ln \left(1 - \frac{(1-p^2)(1-e^{-\beta})}{q}\right)
+ \frac{c}{2q}\, \ln \left(1 - \left[p^2 + \frac{1-p^2}{q}\right]\left(1-e^{-\beta}\right)\right)\, .
\ee

The more complicated term is $G_N^{(1)}$.
Conditioning on the spins at the first level $\{\tau_{\alpha}^1\}$, and all the $K(i)$ values, we get
(using a notation which is clear from the context)
$$
{\E}[\lambda_{\alpha}^{m_2}\, |\, \{K(i)\}_{i=1}^{N},\{\tau^1_{\alpha,i,k}\}_{\alpha,i,k}]\,
=\, \prod_{i=1}^N {\E}^{\{X_{i,k}\},\{\tau_{i,k}^2\}}
\left[\left(\sum_{\sigma_i=1}^{q} \prod_{k=1}^{K(i)} \left[e^{-\beta \delta(\s_i,\tau_{i,k})}\right]\right)^{m_2}\right]\, ,
$$
where the $\tau_{i,k}^2$ are all i.i.d., uniform on $\{1,\dots,q\}$, and
$$
\tau_{i,k} = X_{i,k} \tau_{i,k}^1 + (1-X_{i,k}) \tau_{i,k}^2\, .
$$
The $X_{i,k}$'s are i.i.d., Bernoulli-$p$ random variables.
Since the formulas are identically distributed for different $i$'s and since there are $N$ such $i$'s (canceling the division by $N$),
we get the formula
$$
G_N^{(1)}\,
=\, \frac{1}{m_1} {\E}^{\kappa} \ln
{\E}^{\{\tau_{k}^1\}}\left[
 {\E}^{\{X_{k}\},\{\tau_{k}^2\}}
\left[\left(\sum_{\sigma=1}^{q} \prod_{k=1}^{\kappa} \left[e^{-\beta \delta(\s,\tau_{k})}\right]\right)^{m_2}\right]^{m_1/m_2}
\right]\, ,
$$
where once again $\kappa$ is a Poisson random variable with mean $c$, and now $\{\tau^1_k\}$ and $\{\tau^2_k\}$ are all i.i.d.,
uniform random variables on $\{1,\dots,q\}$, and $\{X_k\}$ are all i.i.d., Bernoulli random variables with mean $p$,
and $\tau_k = X_k \tau^1_k + (1-X_k) \tau^2_k$ for each $k$.

We now take $m_2 \uparrow 1$ and $m_1 \downarrow 0$.
Taking $m_2\uparrow 1$ gives
$$
 \frac{1}{m_1} {\E}^{\kappa} \ln
{\E}^{\{\tau_{k}^1\}}\left[
 {\E}^{\{X_{k}\},\{\tau_{k}^2\}}
\left[\sum_{\sigma=1}^{q} \prod_{k=1}^{\kappa} \left[e^{-\beta \delta(\s,\tau_{k})}\right]\right]^{m_1}
\right]\, ,
$$
which can be rewritten
$$
 \frac{1}{m_1} {\E}^{\kappa} \ln
{\E}^{\{\tau_{k}^1\}}\left[
\left(\sum_{\sigma=1}^{q} \prod_{k=1}^{\kappa} {\E}^{X_{k},\tau_{k}^2} \left[e^{-\beta \delta(\s,\tau_{k})}\right]\right)^{m_1}
\right]\, .
$$
But
$$
{\E}^{X_{k},\tau_{k}^2} \left[e^{-\beta \delta(\s,\tau_{k})}\right]\,
=\,  p e^{-\beta \delta(\s,\tau_{k}^1)} + (1-p) \left(1 - \frac{1-e^{-\beta}}{q}\right)\, .
$$
Using this and taking the limit $m_1 \downarrow 0$ gives
$$
G_N^{(1)}\,
=\, {\E}^{\kappa,\{\tau^1_k\}} \left[
\ln \left(
\sum_{\s=1}^{q} \prod_{k=1}^{\kappa} \left(p e^{-\beta \delta(\sigma,\tau_k^1)} + (1-p) \left(1 - \frac{1-e^{-\beta}}{q}\right)\right)
\right) \right]\, .
$$

Let us now rewrite this in a manner which is appropriate for taking derivatives at $p=0$.
We can write $e^{-\beta \delta(\sigma,\tau_k^1)} = 1 - (1-e^{-\beta}) \delta(\sigma,\tau_k^1)$.
Since the {\em average} value of $\delta(\sigma,\tau_k^1)$ is $1/q$, we may also incorporate that:
$$
e^{-\beta \delta(\sigma,\tau_k^1)}\,
=\, 1 - \frac{1-e^{-\beta}}{q} - (1-e^{-\beta}) (\delta(\sigma,\tau_k^1)-q^{-1})\, .
$$
Then we may rewrite ${\E}^{X_{k},\tau_{k}^2} \left[e^{-\beta \delta(\s,\tau_{k})}\right]$ as
$$
\left(1 - \frac{1-e^{-\beta}}{q}\right) - p (1-e^{-\beta}) (\delta(\sigma,\tau_k^1)-q^{-1})\, .
$$
The formula for $G_N^{(1)}$ is simpler if we introduce a new variable,  $x = (1-e^{-\beta})/(q-e^{-\beta})$.
Therefore, we obtain
\be
\label{eq:G1first}
G_N^{(1)}\, =\, \ln q + c \ln\left(1 - \frac{1-e^{-\beta}}{q}\right)
+ {\E}^{\kappa,\{\tau_k^1\}} \left[
\ln {\E}^{\sigma}\left[
 \prod_{k=1}^{\kappa} \left(1 - px\,  (q\delta(\sigma,\tau_k^1)-1)\right)
\right] \right]\Bigg|_{x = \frac{1-e^{-\beta}}{q - (1-e^{-\beta})}}\, .
\ee

Now we want to consider this formula as a function of $p$ perturbatively near $0$.
We say that the $p=0$ RS ansatz is ``stable to RS perturbations'' if it is a local minimizer of the extended variational principle
in the set of RS ansatze.

Starting from the simpler term, (\ref{eq:G2RS}), we rewrite $G_N^{(2)}$  as
$$
\frac{c}{2}\, \ln\left(1 - \frac{1-e^{-\beta}}{q}\right) + \frac{c(q-1)}{2q}\, \ln \left(1 + \frac{p^2(1-e^{-\beta})}{q - (1-e^{-\beta})}\right)
+\frac{c}{2q}\, \ln \left(1 - \frac{(q-1)(1-e^{-\beta})p^2}{q - (1-e^{-\beta})}\right)\, .
$$
Using $x = (1-e^{-\beta})/(q-e^{-\beta})$ this is simpler:
\be
\label{eq:G2}
G_N^{(2)}\, =\, \frac{c}{2}\, \ln\left(1 - \frac{1-e^{-\beta}}{q}\right) + \frac{c}{2q} \left[(q-1) \ln( 1 + p^2 x) +  \ln(1 - (q-1) p^2 x)\right] \bigg|_{x = \frac{1-e^{-\beta}}{q - (1-e^{-\beta})}}\, .
\ee
This is an even function of $p$, so only even powers will appear.
Taylor expansion shows that
$$
(q-1) \ln (1+p^2 x) + \ln(1 - (q-1) p^2 x)\, =\, - \frac{1}{2}\, q(q-1)p^4x^2\, .
$$
Therefore,
$$
G_N^{(2)}\,
=\, \frac{c}{2}\, \ln\left(1 - \frac{1-e^{-\beta}}{q}\right)
- \frac{c (q-1) p^4 x^2}{4} + O(p^6)\Big|_{x = \frac{1-e^{-\beta}}{q - (1-e^{-\beta})}}\, .
$$
This gives
\be
\label{eq:G2dp}
\frac{d^4}{d p^4} G_N^{(2)}\bigg|_{p=0}\,
=\, - 6 c (q-1) x^2 \Big|_{x = \frac{1-e^{-\beta}}{q - (1-e^{-\beta})}}\, .
\ee

Now turning to the more difficult term, let us start with (\ref{eq:G1first}).
Let us write $f(\sigma,\tau) = (q \delta(\sigma,\tau) - 1)$.
Then we have
\be
\label{eq:G1f}
G_N^{(1)}\, =\, \ln q + c \ln\left(1 - \frac{1-e^{-\beta}}{q}\right)
+ {\E}^{\kappa,\{\tau_k^1\}} \left[
\ln {\E}^{\sigma}\left[
 \prod_{k=1}^{\kappa} \left(1 - px f(\sigma,\tau_k^1)\right)
\right] \right]\Bigg|_{x = \frac{1-e^{-\beta}}{q - (1-e^{-\beta})}}\, .
\ee
As usual, we may interpret the function
$$
{\E}^{\sigma}\left[
 \prod_{k=1}^{\kappa} \left(1 - px f(\sigma,\tau_k^1)\right)
\right]
$$
as a cumulant generating function.
But the random variable is multi-linear in $p$.
Therefore, when expanding in $p$, we have to take account of these terms.
Also, notice that ${\E}^{\sigma}[f(\sigma,\tau)] = {\E}^{\tau}[f(\sigma,\tau)] = 0$
as long as the expectations are with respect to the uniform measure.
Because of this, various terms vanish either in the expectation over ${\E}^{\sigma}$
or in the expectation over ${\E}^{\{\tau_k^1\}}$.

For instance, using the fact that ${\E}^{\sigma}[f(\sigma,\tau)] = 0$, we see that the first derivative in $p$ equals $0$.
Moreover, since each factor is linear in $p$, in taking multiple derivatives (of a single copy of the product) means we cannot repeat
the derivative of any factor.
So we obtain
$$
\frac{d^2}{dp^2}\,
{\E}^{\sigma}\left[
 \prod_{k=1}^{\kappa} \left(1 - px f(\sigma,\tau_k^1)\right)
\right]\, =\, x^2  \sum_{\substack{j,k=1\\ j\neq k}}^{\kappa} {\E}^{\sigma}[f(\s,\tau_{j}^1)f(\s,\tau_k^{1})]\, .
$$
But then taking the expectation over ${\E}^{\{\tau_k^1\}}$ gives $0$ because since $j\neq k$, we have
$$
{\E}^{\{\tau_k^1\}}{\E}^{\sigma}[f(\s,\tau_{j}^1)f(\s,\tau_k^{1})]\,
=\, {\E}^{\sigma}\left[{\E}^{\tau_j^1}[f(\s,\tau_{j}^1)] \cdot {\E}^{\tau_k^1}[f(\s,\tau_k^{1})]\right]\, =\, 0\, .
$$
Continuing, we may easily see that the third derivative is again $0$ since ${\E}^{\sigma}[f(\sigma,\tau)] = 0$.
Then, the next simplest term arises from
\begin{align*}
\frac{d^4}{dp^4}\,
{\E}^{\sigma}\left[
 \prod_{k=1}^{\kappa} \left(1 - px f(\sigma,\tau_k^1)\right)
\right]\,
&=\, x^4  \sum_{\substack{j,k,\ell,m=1\\ j\neq k \neq \ell \neq m}}^{\kappa} {\E}^{\sigma}[f(\s,\tau_{j}^1)f(\s,\tau_k^{1})f(\s,\tau_{\ell}^1)f(\s,\tau_m^{1})]\\
&\quad - 3 x^4 \left(\sum_{\substack{j,k=1\\ j\neq k}}^{\kappa} {\E}^{\sigma}[f(\s,\tau_{j}^1)f(\s,\tau_k^{1})]\right)^2\, .
\end{align*}
We can rewrite this by expanding the square of the sum, and using replicated spin variables for products of expectations:
\begin{align*}
\frac{d^4}{dp^4}\,
{\E}^{\sigma}\left[
 \prod_{k=1}^{\kappa} \left(1 - px f(\sigma,\tau_k^1)\right)
\right]\,
&=\, x^4  \sum_{\substack{j,k,\ell,m=1\\ j\neq k \neq \ell \neq m}}^{\kappa} {\E}^{\sigma}[f(\s,\tau_{j}^1)f(\s,\tau_k^{1})f(\s,\tau_{\ell}^1)f(\s,\tau_m^{1})]\\
&\quad - 3 x^4 \sum_{\substack{j,k=1\\ j\neq k}}^{\kappa} \sum_{\substack{\ell,m=1\\ \ell\neq m}}^{\kappa}{\E}^{\sigma,\sigma'}[f(\s,\tau_{j}^1)f(\s,\tau_k^{1})f(\s',\tau_{\ell}^1)f(\s',\tau_m^{1})]\, .
\end{align*}
Any distinct terms for $j,k,\ell,m$ vanish in the expectation over ${\E}^{\{\tau_k^1\}}$.
Therefore all must be paired. That means that the first summand vanishes entirely.
In the second summand, we require $(\ell,m)=(j,k)$ or $(\ell,m)=(k,j)$.
These two possibilities give an extra factor of $2$.
Hence, we obtain
$$
\frac{d^4}{d p^4} G_N^{(1)} \bigg|_{p=0}\, =\, -6 x^4\, {\E}^{\kappa,\{\tau_k^1\}}  \sum_{\substack{j,k=1\\ j\neq k}}^{\kappa}
\left({\E}^{\sigma}[f(\sigma,\tau_j^1) f(\sigma,\tau_k^1)]\right)^2\Bigg|_{x = \frac{1-e^{-\beta}}{q - (1-e^{-\beta})}}\, .
$$
A calculation gives
$$
{\E}^{\sigma}[f(\sigma,\tau_j^1) f(\sigma,\tau_k^1)]\,
=\, \begin{cases} - 1 & \text{ if $\tau_j^1 \neq \tau_k^1$,}\\
(q-1) & \text { if $\tau_j^1 = \tau_k^1$.}
\end{cases}
$$
Using the i.i.d., uniform distribution on $\{\tau_k^1\}$ gives ${\bf P}\{\tau_j^1 = \tau_k^1\} = 1/q$.
Therefore,
$$
{\E}^{\{\tau_k^1\}} {\E}^{\sigma} [f(\sigma,\tau_j^1) f(\sigma,\tau_k^1)]\, =\, 0\, ,
$$
as we claimed before.
But now we also have
$$
\left({\E}^{\sigma}[f(\sigma,\tau_j^1) f(\sigma,\tau_k^1)]\right)^2\,
=\, \begin{cases}  1 & \text{ if $\tau_j^1 \neq \tau_k^1$,}\\
(q-1)^2 & \text { if $\tau_j^1 = \tau_k^1$,}
\end{cases}
$$
which gives
$$
{\E}^{\{\tau_k^1\}} \left[
\left({\E}^{\sigma}[f(\sigma,\tau_j^1) f(\sigma,\tau_k^1)]\right)^2 \right]\, =\, q-1\, .
$$
Therefore, also using the fact that ${\E}^{\kappa} [\#\{(j,k) \in \{1,\dots,\kappa\}^2\, :\, j\neq k\}]$
equals ${\E}^{\kappa}[\kappa(\kappa-1)] = c^2$, we obtain
\be
\label{eq:G1dp}
\frac{d^4}{d p^4} G_N^{(1)} \bigg|_{p=0}\, =\, - 6 c^2 (q-1) x^4\Big|_{x = \frac{1-e^{-\beta}}{q - (1-e^{-\beta})}}\, .
\ee

\section*{Acknowledgements}

We thank Alessandra Bianchi and Anton Bovier for useful discussion on the ferromagnetic
version of the model.
We thank S. Franz, F. Krzakala and A. Montanari for several useful observations. 
P.C. thanks {\em Strategic Research Grant} (University of Bologna). 
The work of S.D. is supported in part by The Netherlands Organisation for
Scientific Research (NWO). C.G. acknowledges {\em International Research Projects}
(Fondazione Cassa di Risparmio and University of Modena) and {\em FIRB
project} (grant n. RBFR10N90W) for financial support.
The work of S.S. is supported by an NSA Young Investigators grant.

\end{document}